\documentclass[12pt]{article}
\usepackage[utf8]{inputenc}
\usepackage[pdftex]{graphicx} 
\usepackage{amsmath}   
\usepackage{amssymb}   
\usepackage{amsthm}
\usepackage{amsfonts}
\usepackage{geometry} 
\usepackage{booktabs}
\usepackage{enumerate}
\usepackage{subfig}
\usepackage{url}
\usepackage{fancyhdr}
\usepackage{layout}
\usepackage{pinlabel}
\usepackage[normalem]{ulem}

\usepackage[pdftex]{color}
\usepackage{xcolor}

\usepackage[parfill]{parskip}
\makeatletter
\def\thm@space@setup{%
  \thm@preskip=\parskip \thm@postskip=0pt
}
\makeatother

\numberwithin{equation}{section}
\newtheorem{de}{Definition}[section]
\newtheorem{theo}[de]{Theorem}

\newtheorem{lemm}[de]{Lemma}
\newtheorem{co}[de]{Corollary}
\newtheorem{theol}{Theorem}

\newtheorem{re}[de]{Remark}

\theoremstyle{remark}

\theoremstyle{plain}
\newtheorem*{no}{Notation}

\newcommand{\vungoc}{V\~u Ng\d{o}c }
\newcommand{\nff}{{n_{\text{FF}}}}

\newcommand{\dee}{\mathrm{d}}

\newcommand{\N}{\mathbb{N}}

\newcommand{\R}{\mathbb{R}}
\newcommand{\mbS}{\mathbb{S}}
\newcommand{\T}{\mathbb{T}} 
\newcommand{\Z}{\mathbb{Z}}

\newcommand{\eps}{\varepsilon}
\newcommand{\ze}{\zeta}

\newcommand{\ka}{\kappa}
\newcommand{\lam}{\lambda}

\newcommand{\om}{\omega}

\newcommand{\Ga}{\Gamma}
\newcommand{\De}{\Delta}

\newcommand{\Lam}{\Lambda}

\newcommand{\zeti}{{\tilde{\ze}}}

\newcommand{\That}{{\hat{T}}}

\newcommand{\What}{{\hat{W}}}

\newcommand{\mcA}{\mathcal A}

\newcommand{\mcF}{\mathcal F}
\newcommand{\mcG}{\mathcal G}

\newcommand{\mcJ}{\mathcal J}

\newcommand{\mcN}{\mathcal N}

\newcommand{\mcT}{\mathcal T}

\newcommand{\mcW}{\mathcal W}
\newcommand{\mcX}{\mathcal X}

\newcommand{\mfB}{\mathfrak B}

\newcommand{\mfH}{\mathfrak H}
\newcommand{\mfI}{\mathfrak I}
\newcommand{\mfJ}{\mathfrak J}

\newcommand{\mfL}{\mathfrak L}

\newcommand{\mfP}{\mathfrak P}

\newcommand{\mfT}{\mathfrak T}

\newcommand{\mfW}{\mathfrak W}

\newcommand{\mfh}{\mathfrak h}

\newcommand{\mfj}{\mathfrak j}
\newcommand{\mfk}{\mathfrak k}
\newcommand{\mfl}{\mathfrak l}

\newcommand{\mfn}{\mathfrak n}

\newcommand{\mfp}{\mathfrak p}

\newcommand{\mfw}{\mathfrak w}

\title{\textbf{Taylor series and twisting-index invariants of coupled spin-oscillators}}
\author{Jaume Alonso \and Holger R.\ Dullin \and Sonja Hohloch}

\begin{document}

\maketitle

\begin{abstract}
\noindent About six years ago, semitoric systems on 4-dimensional manifolds were classified by Pelayo \&  V\~u Ng\d{o}c by means of five invariants. A standard example of such a system is the coupled spin-oscillator on $\mbS^2 \times \R^2$. Calculations of three of the five semitoric invariants of this system (namely the number of focus-focus singularities, the generalised semitoric polygon, and the height invariant) already appeared in the literature, but the so-called twisting index was not yet computed and, of the so-called Taylor series invariant, only the linear terms were known.\\

\noindent In the present paper, we complete the list of invariants for the coupled spin-oscillator by calculating higher order terms of the Taylor series invariant and by computing the twisting index. Moreover, we prove that the Taylor series invariant has certain symmetry properties  that make the even powers in one of the variables vanish and allow us to show superintegrability of the coupled spin-oscillator on the zero energy level.
\end{abstract}


\section{Introduction}
Completely integrable Hamiltonian systems have been instrumental in the development of modern dynamical systems. Even though they are very special systems in the class of all Hamiltonian systems, many 
fundamental examples with great physical importance are of this type. Presently we are nowhere close to a global
classification theory of integrable systems. However, such a classification is possible if we add certain additional restrictions, foremost on the type of singularities 
that the systems are allowed to have. A classical case is that 
of toric systems, where only elliptic singularities appear. 

Semitoric systems, which in addition allow for 
so-called focus-focus singularities, have been recently classified in dimension 4 by Pelayo \&  V\~u Ng\d{o}c \cite{PV1}. Together with certain further 
assumptions, a semitoric system is characterised by a list of five invariants (see below), such that two systems are equivalent
when their invariants agree. Computing these invariants for important examples is an ongoing programme and
the best understood example is the spin-oscillator, see Pelayo \&  V\~u Ng\d{o}c \cite{PV3}, which is a special case of the Jaynes-Cumings model (cf.\ Babelon \& Cantini \& Douçot \cite{BCD})
describing the interaction of light and matter.

One of the five invariants of a semitoric system is the Taylor series invariant, which captures the behaviour 
of the system near the separatrix of a focus-focus point. This invariant is notoriously difficult to compute,
and up to now nonlinear terms of this invariant have only been computed for the spherical pendulum, cf.\ Dullin \cite{Du}.
The spherical pendulum  is however not in the class of semitoric systems for which the global classification by Pelayo \&  V\~u Ng\d{o}c \cite{PV1} applies, since it fails to satisfy certain compactness requirements. In the forthcoming work \cite{ADH2}, the authors of the present paper also compute nonlinear terms of the Taylor series invariant of the coupled angular momenta, together with other symplectic invariants, and study how these depend on the three parameters of the family. 

The goal of this paper is to compute some 
nonlinear terms of the Taylor series invariant of the spin-oscillator and also the so-called twisting-index invariant, making it the first semitoric system for which the complete list of invariants is known.

\subsubsection*{Setting and conventions}

Throughout the paper, let $(M,\omega)$ be a 4-dimensional connected symplectic manifold.

 To any smooth function $f:M \to \R$, the \emph{Hamiltonian vector field} $\mcX_f$ is associated via $\omega(\mcX_f,\cdot) = -df$. The flow of the Hamiltonian vector field is called the \emph{Hamiltonian flow}. If $f,g:M \to \R$ are two smooth functions, their \emph{Poisson bracket} is defined by $\{f,g\}:=\omega(\mcX_f,\mcX_g) = -df(\mcX_g) = dg(\mcX_f)$. The functions $f$ and $g$ are said to \emph{Poisson-commute} if $\{f,g\}=0$, which means that each function is constant along the Hamiltonian flow lines of the other one.

Let $(L,H):M \to \mathbb{R}^2$ be a pair of smooth functions. We say that the triple $(M,\omega,(L,H))$ is a \emph{completely integrable system} with two degrees of freedom if the functions $L$ and $H$ Poisson-commute and their differentials $dL$, $dH$ are almost everywhere linearly independent. Therefore the \emph{momentum map} $F:=(L,H): M \to \R^2$ induces a fibration on $M$.

The points where the differentials $dL$, $dH$ fail to be linearly independent, i.e., $dF$ has not maximal rank, are called \emph{critical points} or \emph{singularities}. Noncritical points are called \emph{regular}. Compact connected fibres consisting entirely of regular points are called \emph{regular} and form so-called \emph{Liouville tori} where the behaviour of the system is simple as described by the \emph{Arnold-Liouville theorem} (see Arnold \cite{Ar}). As a consequence, distinguishing symplectic-dynamical properties of the system must be encoded in the \emph{singular} fibres, i.e., fibres containing at least one singularity.

The singularities of completely integrable systems in $2n$ dimensions have been characterised by Eliasson \cite{El1,El2} and Miranda \& Zung \cite{MZ} by means of normal forms. In particular, if we restrict us to the \emph{4-dimensional} case and to non-degenerate singularities (see Bolsinov \& Fomenko \cite{BoF} or Vey \cite{Ve} for a precise definition), given a singularity $m \in M$, there exist local symplectic coordinates $(x_1,y_1,x_2, y_2)$ centered at $m$ and functions $(Q_1,Q_2)$ satisfying $\{L,Q_i\}=\{H,Q_i\}=0$ for $i=1,2$ of the following types:
\begin{quote}
\begin{enumerate}[a)]
	\item \textit{Elliptic component:} $Q_i(x_1,y_1,x_2, y_2)=({x_i}^2+{y_i}^2)/2$.
	\item \textit{Hyperbolic component:} $Q_i(x_1,y_1,x_2, y_2)=x_i y_i$.
	\item \textit{Focus-focus components (always come in pairs): }
	
	$Q_1(x_1,y_1,x_2, y_2)=x_1 y_1 + x_2 y_2$ \textit{and} $Q_2(x_1,y_1,x_2, y_2) = x_1 y_2 - x_2 y_1.$	
	\item \textit{Regular component:} $Q_i(x_1,y_1,x_2, y_2) = y_i$.
\end{enumerate}
\end{quote}
where the case a), b), and d) can be mixed to obtain two functions $Q_1$ and $Q_2$.

A \emph{semitoric system} is a four-dimensional completely integrable system $(M,\omega,(L,H))$ with two degrees of freedom where all singularities are non-degenerate, have no hyperbolic components and the map $L$ induces a faithful Hamiltonian ${\mathbb S}^1$-action on $M$ and is \emph{proper} (i.e., the preimage by $L$ of a compact set is compact in $M$). In particular, the Hamiltonian flow of $L$ is $2\pi$-periodic.

Excluding hyperbolic components, the singularities of semitoric systems can only be the combinations \emph{elliptic-elliptic}, \emph{regular-elliptic} (often called \emph{transversally elliptic}) or \emph{focus-focus} type depending on whether they have two elliptic components, one elliptic component and one regular component, or coupled focus-focus components respectively.

Semitoric systems appear often in physics and have attracted the attention of mathematicians in the last years, see Pelayo \& V\~u Ng\d{o}c \cite{PV2} for an overview. They have recently given  a symplectic classification of semitoric systems in terms of the following five symplectic invariants, cf.\ Pelayo \& V\~u Ng\d{o}c  \cite{PV1}, \cite{PV4}:

\begin{quote}
\begin{enumerate}[(1)]
	\item The \textit{number of focus-focus singularities}, denoted by $\nff$.
	\item \textit{Polygon invariant:} an equivalence class of labelled collections of rational convex polygons and vertical lines crossing them, cf.\ section \ref{sec:defpolygon} for more details.
	\item \textit{Height invariant:} $\nff$ numbers corresponding to the height of the focus-focus critical values in the rational convex polygons, cf.\ section \ref{sec:defheight} for more details.
	\item \textit{Taylor series invariant:} a collection of $\nff$ formal Taylor series in two variables describing the foliation around each focus-focus singular fiber, cf.\ section \ref{sec:deftaylor} for more details.
	\item \textit{Twisting-index invariant:} $\nff$ integers measuring the twisting of the system around singularities, cf.\ section \ref{sec:deftwistingIndex} for more details.
\end{enumerate}
\end{quote}

Two semitoric systems $(M_1,\omega_1,(L_1,H_1))$ and $(M_2,\omega_2,(L_2,H_2))$ are said to be \emph{isomorphic} if there exists a symplectomorphism $\varphi:M_1 \to M_2$ such that $\varphi^*(L_2,H_2) = (L_1,f(L_1,H_1))$, where $f$ is some smooth function such that $\frac{\partial f}{\partial H_1}>0$. The importance of the symplectic classification lies in the fact that two semitoric systems are isomorphic if and only if their five symplectic invariants coincide. Moreover, given an `admissible' list of invariants, the corresponding semitoric system can be constructed.

Probably the simplest example of a non-compact semitoric systems is the coupled spin-oscillator, consisting of the coupling of a classical spin on the 2-sphere ${\mathbb S}^2$ with a harmonic oscillator in the plane $\mathbb{R}^2$, 
as studied by Pelayo \& \vungoc in \cite{PV3}. It is a simplification of the Jaynes-Cummings model, see Babelon \& Cantini \& Douçot \cite{BCD} for a recent study. 

Consider $M=\mbS^2 \times \R^2 $ and let $(x,y,z)$ be Cartesian coordinates on the unit sphere ${\mathbb S}^2 \subset \mathbb{R}^3$ and $(u,v)$ Cartesian coordinates on the plane $\mathbb{R}^2$. 
Endow $M$ with the symplectic form $\omega = \lambda\, \omega_{{\mathbb S}^2} \oplus \mu\, \omega_{\mathbb{R}^2}$, 
where $\omega_{{\mathbb S}^2}$ and $\omega_{\mathbb{R}^2}$ are the standard symplectic structures on ${\mathbb S}^2$ and $\mathbb{R}^2$ respectively and $\lam,\mu>0$ are positive constants. 
The \emph{spin-oscillator} is a 4-dimensional Hamiltonian integrable system $(M,\omega,(L,H))$, where the momentum map  $F=(L, H): M \to \mathbb{R}  ^2$ is given by
\begin{equation*}
L(x,y,z,u,v):= \mu \dfrac{u^2+v^2}{2} + \lambda (z-1) \quad \text{and} \quad H(x,y,z,u,v):=\dfrac{xu+yv}{2}.
\end{equation*}

The map $L$ is the momentum map for the simultaneous rotation of the sphere around its vertical axis and the plane around the origin. The image of the map $F$ is displayed in Figure \ref{figmom}. The system has exactly one focus-focus singularity, one elliptic-elliptic singularity and two one-parameter families of transversally elliptic singularities.

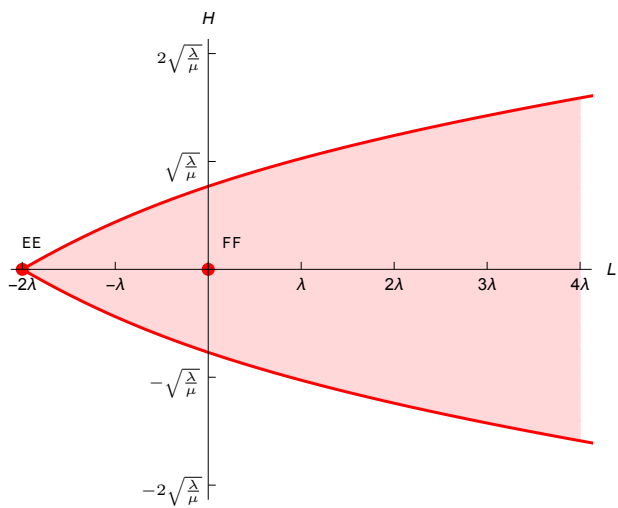
\begin{figure}[ht]
 \centering
\labellist
 \tiny \hair 2pt
 \pinlabel $-2\sqrt{\frac{\lambda}{\mu}}$ at 63 4
 \pinlabel $-\sqrt{\frac{\lambda}{\mu}}$ at 64 46
 \pinlabel $\sqrt{\frac{\lambda}{\mu}}$ at 67 130
 \pinlabel $2\sqrt{\frac{\lambda}{\mu}}$ at 66 172
\endlabellist
\includegraphics[width=8cm]{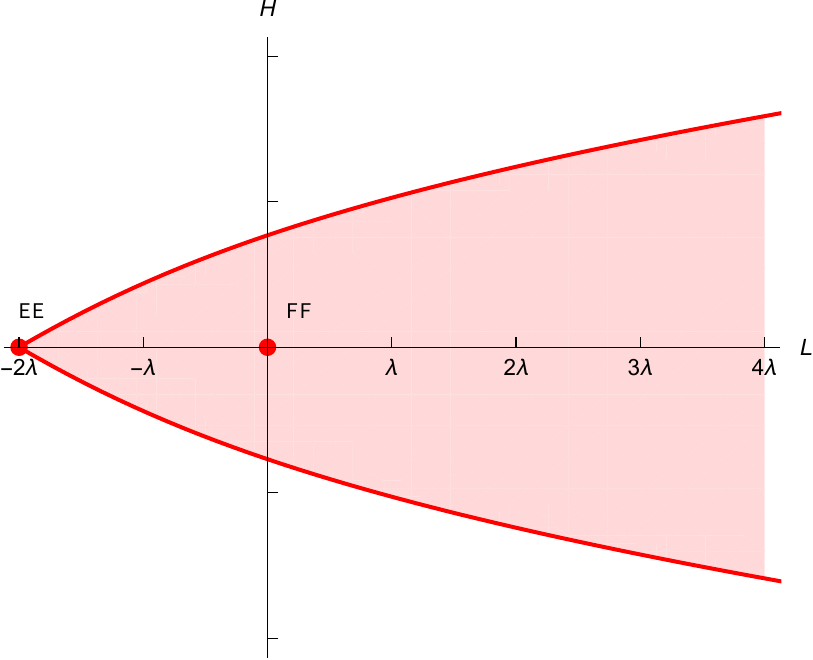}
 \caption{\small Image of the momentum map $F=(L,H)$. The spin-oscillator has one singularity of elliptic-elliptic type at the critical value $(-2\lam,0)$ and one of focus-focus type at $(0,0)$.}
 \label{figmom}
\end{figure}

\subsubsection*{Main results}

The symplectic invariants of the spin-oscillators have been explicitely calculated by Pelayo \& V\~u Ng\d{o}c in \cite{PV3} except for the Taylor series invariant, which is only given up to linear order, and the twisting-index invariant. The main result of the present paper is the following:

\begin{theol}
Let $l$ be the value of the integral $L$ and $j$ be the value of Eliasson's $Q_1$ function.
The leading order terms of the Taylor series invariant of the spin-oscillator 
are
\begin{align*}
S(l,j) &= \dfrac{\pi}{2} l + (5 \log 2+ \log \lambda) j + \dfrac{1}{4\lambda} l j - \dfrac{1}{768\lambda^2}j (39 {l}^2 + 34 {j}^2) + \dfrac{1}{1536 \lambda^3}j(34l {j}^2 + 23 {l}^3) \nonumber \\
&- \dfrac{1}{2621440 \lambda^4}j (13505 {l}^4+30620 {l}^2{j}^2+10727 {j}^4) + ...,
\end{align*}
and $\partial S/\partial j$ is an even function of $j$.
\end{theol}

Theorem A is restated and proven in Theorem \ref{inv1} and Theorem \ref{noeven}. The second author of the present paper calculated the Taylor series invariant for the spherical pendulum in \cite{Du}. The spherical pendulum is a completely integrable system which also has a simple focus-focus point, but it is in fact a \emph{generalised} semitoric system, since the angular momentum integral fails to be proper (see Pelayo \& V\~u Ng\d{o}c \& Ratiu \cite{PRV} for details on generalised semitoric systems). In this paper, we apply similar ideas to calculate the higher order terms of the Taylor series invariant of the coupled spin-oscillators, which is a semitoric system.

To our knowledge this is the first time that the symplectic Taylor series  invariant has been explicitly calculated with its higher order terms for a semitoric system. 
Once the Taylor series invariant has been computed it is straightforward to obtain series expansion for derived dynamical 
quantities like the reduced period, the rotation number, and the twist, i.e.\ the derivative of the rotation number with respect 
to circular action at constant energy, as has been done in Dullin \cite{Du} for the spherical pendulum. In the upcoming work \cite{ADH2}, the authors compute the Taylor series invariant with higher order terms of the coupled angular momenta, a family of semitoric systems depending on three parameters.

Due to a discrete symmetry of the spin-oscillator we are also able to show the following:

\begin{theol}
The twist of the spin-oscillator vanishes on the energy surface
$H=0$.
\end{theol}

Theorem B is restated in Theorem \ref{theotwis}. We show that this implies that there is an additional integral, which we give explicitly. As a consequence, the spin-oscillator is super-integrable on the energy surface $H=0$.

Up to our knowledge, the twisting index has never been computed explicitly anywhere in the literature.  For the coupled spin-oscillator, we obtain the following:

\begin{theol}
The twisting-index invariant of the coupled spin-oscillator system consists of the association of indices $k$ to each of the weighted polygons of the polygon invariant as represented in Figure \ref{twistbig}. 

\end{theol}

Theorem C is restated more precisely and proven in Theorem \ref{inv2}. The twisting-index invariant has a stronger meaning for systems with more than one focus-focus singularity, since it allows for comparison of the relative twisting of different singularities. Nonetheless, it is defined for systems with just one focus-focus singularity, too. The authors also compute the twisting-index invariant of the coupled angular momenta, a family of semitoric systems with also one focus-focus point, in the subsequent work \cite{ADH2}. Very recently, an explicit family of compact semitoric systems admitting two focus-focus singularities was found by Hohloch $\&$ Palmer \cite{HP}.

\begin{figure}[ht!]
\centering
\subfloat[$\epsilon=+1$, $k=-1$]{
  \includegraphics[width=0.35\textwidth]{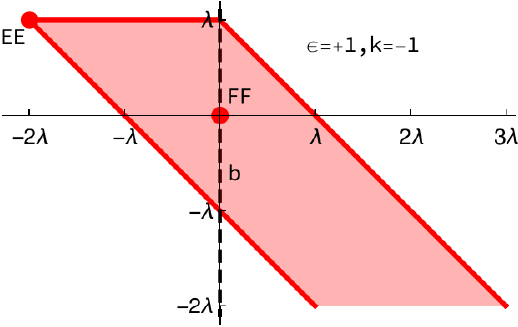}
}
\subfloat[$\epsilon=+1$, $k=0$]{
  \includegraphics[width=0.35\textwidth]{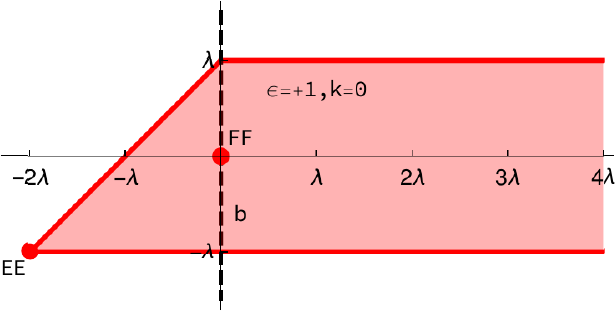}
}
\subfloat[$\epsilon=+1$, $k=+1$]{
  \includegraphics[width=0.25\textwidth]{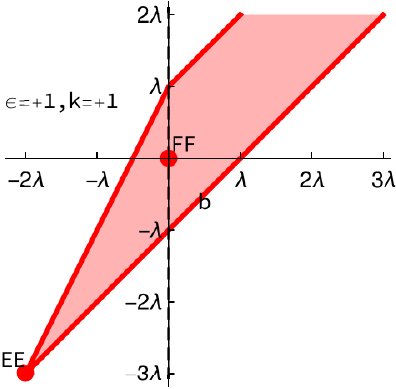}
}
\newline
\subfloat[$\epsilon=-1$, $k=-1$]{
  \includegraphics[width=0.35\textwidth]{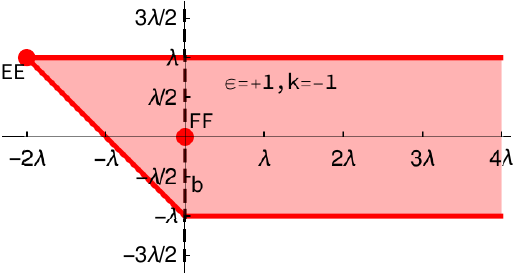}
}
\subfloat[$\epsilon=-1$, $k=0$]{
  \includegraphics[width=0.35\textwidth]{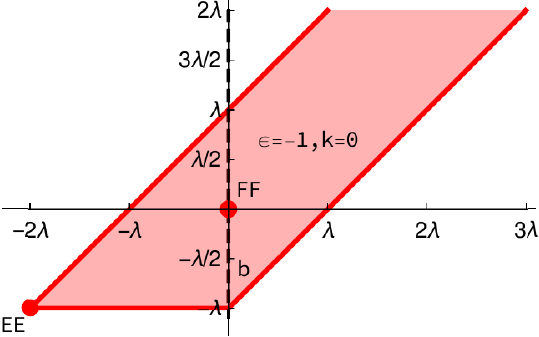}
}
\subfloat[$\epsilon=-1$, $k=+1$]{
  \includegraphics[width=0.2\textwidth]{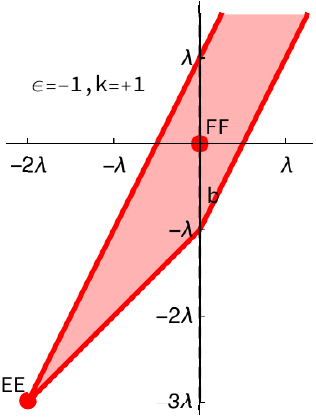}
}
    \caption{Representation of the twisting-index invariant of the coupled spin-oscillator. The polygon invariant consists of an infinite family of weighted polygons, some of which are depicted. The twisting-index invariant consists of the association of the index $k$ to each of the polygons of the polygon invariant. Note that polygons in the same row are related by an integral-affine transformation. Polygons in the same column correspond to different choices of cutting direction $\epsilon$ and therefore have the same index.}
    \label{twistbig}
\end{figure} 

\subsubsection*{Structure of the paper}

In section 2, we briefly recall  the definition of the semitoric invariants.
In section 3, we rewrite the coupled spin-oscillator in suitable local coordinates, recall some definitions and properties of elliptic integrals and compute the action integral of the system.
In section 4, we compute the Taylor series invariant. In section 5, we define and calculate the period and rotation number. In section 6, we study the twist and prove superintegrability on the energy surface $H=0$. In section 7, we compute the twisting index.


\subsubsection*{Figures} 

Figures 1, 2, 5 and 6 have been made with {\em Mathematica} and Figures 3 and 4 have been made with {\em Inkscape}.


\subsubsection*{Acknowledgements}

We wish to thank Yohann Le Floch, Joseph Palmer, San V\~u Ng\d{o}c, and Daniele Sepe for helpful discussions and the referee for carefully reviewing our manuscript. Moreover, the first author has been fully and the third author partially funded by the Research Fund of the University of Antwerp.


\section{Semitoric invariants}

In this section, we briefly recall the definition of the semitoric invariants restricting ourselves to the case of systems with one singularity of focus-focus type. We start by the Taylor series invariant, since it is the only \emph{semi-global} invariant, i.e.\ the only one that exclusively depends on the characteristics of the system in a neighbourhood of the critical fibre and therefore can be defined in more general classes of systems (see Pelayo \& V\~u Ng\d{o}c \& Ratiu \cite{PRV}). After that we continue with the number of focus-focus points invariant, the polygon invariant, the height invariant and the twisting-index invariant. These last four invariants  depend on the properties of the whole system and therefore are said to be \emph{global}. For more details on the definitions we refer to Pelayo \& \vungoc in \cite{PV1} and \cite{PV4}. 

In the whole section, let $(M,\omega,F:=(L,H))$ be a semitoric system with only one singularity $m \in M$ of focus-focus type. Assume, adding a constant to $F$ if needed, that $c:=F(m)=0$ and that $m$ is the only critical point of the singular fibre $F^{-1}(m)$, which is a generic condition (see Pelayo \& \vungoc \cite{PV1}). We denote the image of the momentum map by $B:=F(M) \subseteq \R^2$. The set of regular values of $F$ is then $B_r$ $:=$ Int$(B) \backslash \{c\}$.


\subsection{The definition of the Taylor series invariant}
\label{sec:deftaylor}

Let $\mcF$ be the associated singular foliation by the components of $F$. The neighbourhood of the critical point $m$ can be described using the normal forms by Eliasson \cite{El1,El2} and Miranda \& Zung \cite{MZ}, as mentioned in the introduction. There exist local symplectic coordinates $(x_1,y_1,x_2,y_2)$ around $m$ in which the foliation $\mcF$ is given by the level sets of the function $Q:=(Q_1,Q_2)$ with
\begin{equation}
Q_1(x_1,y_1,x_2,y_2):=x_1 y_1 + x_2 y_2,\qquad Q_2(x_1,y_1,x_2,y_2) := x_1 y_2 -  x_2y_1.
\label{norm}
\end{equation}

Then there is a local diffeomorphism $\phi$ of $\mathbb{R}^2$ such that $Q=\phi \circ F$. We can use this diffeomorphism to extend $Q$ to a global momentum map $\Phi := \phi \circ F$ for the whole foliation, which agrees with $Q$ on the neighbourhood of $m$. Define $\Phi := (\Phi_1,\Phi_2)$ and $\Lambda_z := \Phi^{-1}(z)$, so that the singular fibre of $m$ is $\Lambda_0$. From the form  \eqref{norm} we can see that close to the critical point, the $\Phi_2$-orbits must be $2\pi$-periodic  for any point of a  non-trivial trajectory generated by $\Phi_1$.

\begin{figure}[ht]
 \centering
 \includegraphics[width=9cm]{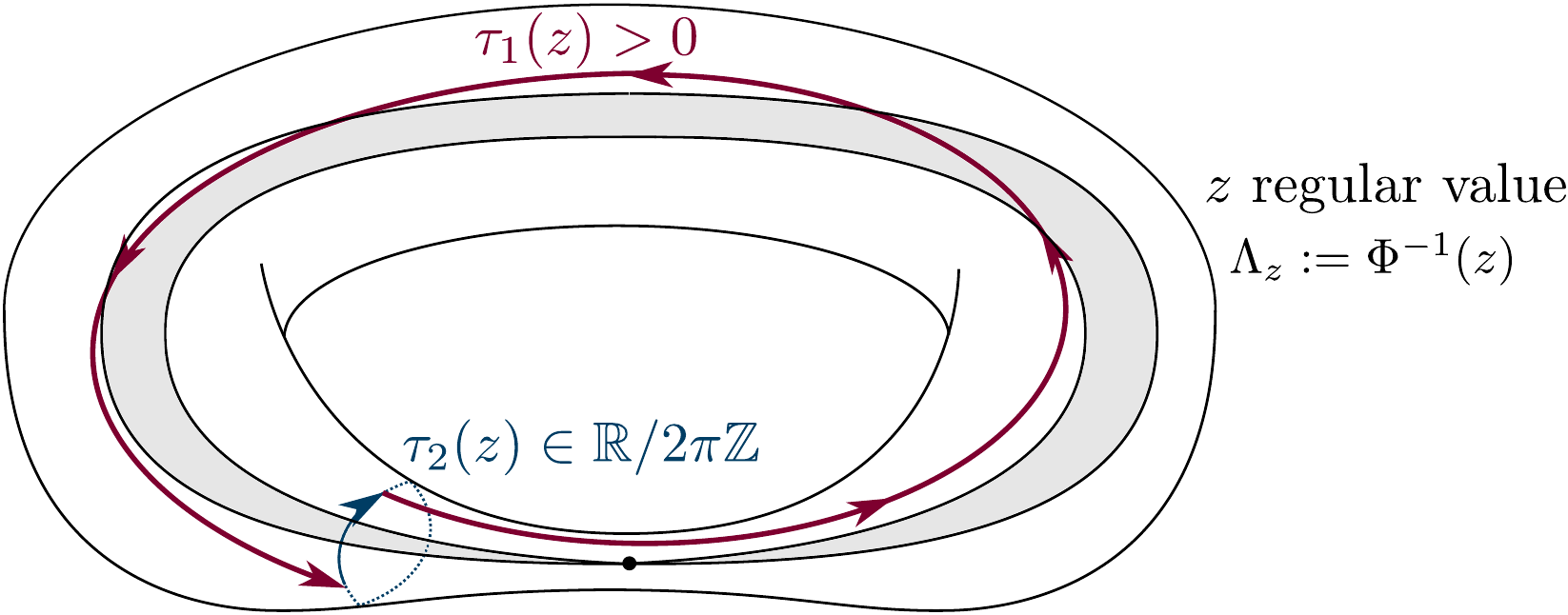}
 \caption{\small A regular fibre $\Lambda_z$ close to the singular fibre $\Lambda_0$ (gray). First we follow the flow generated by $\Phi_1$ (red) and then the flow generated by $\Phi_2$ (blue).}
 \label{cycles2}
\end{figure}

To define the symplectic invariant we need to work with the period lattice. Let us write $z$ as a complex variable $z = (z_1,z_2) = z_1 + \text{i} z_2$. Consider the fibre $\Lambda_z$ of a regular value $z$ and take a point $a \in \Lambda_z$.  From $a$ we may follow the Hamiltonian flow generated by $\Phi_1$ until we reach again the $\Phi_2$-orbit that passes through $a$. Once we reach this point, we may come back to $a$ following the Hamiltonian flow generated by $\Phi_2$. Define $\tau_1(z)>0$ as the time needed for the first displacement and $\tau_2(z) \in \mathbb{R}/2\pi \mathbb{Z} $ the time for the second displacement. Define also
\begin{equation}
\begin{cases}
\sigma_1 (z): = \tau_1(z) + \Re (\log z)\\
\sigma_2 (z): = \tau_2(z) - \Im(\log z)
\end{cases}
\label{taus}
\end{equation} where $\Re$ and $\Im$ represent the real and imaginary parts respectively and $\log$ is a complex logarithm that is smooth near $z$ and has a cut along the positive real axis. For later convenience, we choose the lift of $\tau_2$ to $\R$ that satisfies $\sigma_2(0)\in[0,2\pi[$ and keep this determination throughout the paper. V\~u Ng\d{o}c proved in \cite{Vu1} that $\sigma_1$ and $\sigma_2$ extend to smooth single-valued functions around 0 and that the differential one-form
\begin{equation}
\sigma := \sigma_1 dz_1+ \sigma_2 dz_2
\label{sigma}
\end{equation}
is closed. Then we come to the definition of the Taylor series invariant:
\begin{de}[From V\~u Ng\d{o}c \cite{Vu1}]
Let $S$ be the unique smooth function defined around $0 \in \mathbb{R}^2$ such that \begin{equation}
 \begin{cases}
 \dee S =\sigma\\
S(0)=0,
 \end{cases}
 \end{equation} where $\sigma$ is the one form given by \eqref{sigma}. The Taylor series of $S$ at $(0,0)$ is denoted by $(S)^\infty$. We say that $(S)^\infty$ is the Taylor series invariant of $(M,\omega,F)$ at the focus-focus point $m$.
\end{de}

Note that with our choice of determination of the log function, $\frac{\partial S}{\partial z_2}(0) \in [0,2\pi[$. When it comes to the computation of the Taylor series invariant, a particularly useful fact is that the function $S$ is related to the action $\mcA$ of the system. Let $\varpi$ be a semiglobal primitive of the symplectic form $\omega$, $\beta_z \subset \Lambda_z$ the trajectory used in the definition of the $\tau_i$ and $\mcA(z) := \oint_{\beta_z} \varpi$, then
\begin{equation}
S(z) = \mcA(z) - \mcA(0)+ \Re (z \log z - z).
\label{inv5}
\end{equation}
In these terms, $S(z)$ can be seen as a regularised or desingularised area (Pelayo \& V\~u Ng\d{o}c \cite{PV2}).

\begin{re} Besides the construction that we have seen in terms of the periods of the torus action on fibres close to the critical fibre, the Taylor series invariant admits other interpretations:
\begin{itemize}
	\item On a regular fibre $\Lam_z$, the first component of the system generates a $2\pi$-periodic flow but the second generates an arbitrary flow that turns indefinitely around the focus-focus singularity. As $z \to 0$, the time $\tau_1(z)$ that the flow of the second component needs to perform a loop grows at a logarithmic rate up to a certain `error term'. This error term is given by the symplectic invariant $(S)^\infty$.
	\item From a symplectic point of view, $(S)^\infty$ is also the germ of the (local) generating function of the (singular) Lagrangian fibration induced by $(L,H): M \to \R^2$.
\end{itemize}
\end{re}

\subsection{The number of focus-focus points invariant}
The number $\nff \in \N \cup \{0\}$ of singularities of focus-focus type is a symplectic invariant of the semitoric system. In our case, $\nff=1$.

\subsection{The definition of the polygon invariant} 

\label{sec:defpolygon}

The plane $\R^2$ has a standard integral affine structure defined by the group of integral-affine transformations Aff$(2,\Z)$ $ := $ GL$(2,\Z) \ltimes \R^2$ and $B$ has a natural affine structure induced by $F$. Consider now $\mcJ\subset $Aff$(2,\Z)$, the subgroup of transformations consisting of vertical translations composed with a transformation $T^k$, $k \in \Z$, where
\begin{equation}
 T^k := \begin{pmatrix}
1& 0 \\ k & 1
\end{pmatrix} \in \text{GL}(2,\Z).
\label{TT}
\end{equation} The transformations in $\mcJ$ leave vertical lines invariant. Let us denote the vertical lines by $b_\ka:=\{(\ka,y) \,|\, y \in \R\} \subset \R^2$. For any  $n \in \Z$, we define $t^n_{b_\ka}:\R^2 \to \R^2$ as the transformation consisting of the identity on the halfplane left of $b_\ka$ and $T^n$ on the halfplane right of $b_\ka$. Let $\ka := \pi_1(c)$ be the  first components of the critical values, where $\pi_1:\R^2 \to \R$ is the canonical projection onto the first coordinate. Given a sign $\epsilon \in \{-1,+1\}$, let $b_{\ka}^{\epsilon} \subset b_{\ka}$ be  the half-line that starts in $c$ and extends upwards if $\epsilon=+1$ or downwards if $\epsilon=-1$. 

We say that a subset $\De \subseteq \R^2$ is a \emph{convex polygon} if it is the intersection of closed half-planes, possibly infinite, such that on each compact subset of $\De$ there are at most a finite number of corner points. If furthermore the slopes of the edges meeting at each vertex are rational numbers, we say that $\De$ is {\em rational}. We denote  the space of rational convex polygons in $\R^2$ by Polyg$(\R^2)$ and the set of vertical lines in $\R^2$ by
$$ \text{Vert}(\R^2) := \{ b_\ka\,|\, \ka \in \R\}.$$

\begin{theo}[V\~u Ng\d{o}c \cite{Vu2}, Theorem 3.8]
For $\epsilon \in \{-1,+1\}$ there exists a homeomorphism $f = f_\epsilon : B \to \R^2$, unique modulo a left composition by a transformation in $\mcJ$, such that:
\begin{itemize}
	\item $f |_{(B \backslash b^{\epsilon}_\ka)}$ is a diffeomorphism into the image $\De :=f(B)$ of $f$.
	\item $\De$ is a rational convex polygon. 
	\item $f|_{(B_r\backslash b^{\epsilon}_\ka)}$ is affine, i.e.\ it sends the integral affine structure of $B_r$ to the standard affine structure of $\R^2$.
	\item $f$ preserves $L$, i.e.\ $f(l,h) = (f^{(1)}(l,h), f^{(2)}(l,h))= (l, f^{(2)}(l,h))$.
\end{itemize}
\label{thmpol} 
\end{theo}

We see that the definitions of $f$ and $\De$ are unique up to two choices:
\begin{itemize}
	\item The sign $\epsilon$. A different choice $ \epsilon'$ changes $f$ by $f' = t_{u} \circ f$ and $\De$ by $\De' = t_{u} (\De)$, where $u = (\epsilon-\epsilon')/2$.
	\item A left composition by an element of $\mcJ$, which corresponds to a different initial set of action variables (V\~u Ng\d{o}c \cite{Vu2}, step 2, proof of Theorem 3.8).
\end{itemize}

In order to take into account this freedom when constructing the polygon invariant, we define a \emph{weighted polygon as a triple of the form
$$ \De_\text{weight} = (\De,b_{\ka},\epsilon)$$
where $\De \in \text{Polyg}(\R^2)$, $b_{\ka} \in \text{Vert}(\R^2)$ and $\epsilon \in \{-1,+1\}$. We denote by $\mcW$Polyg$(\R^2)$ the space of all weighted polygons of complexity one. Write now $\mathbb{Z}_2:=\Z \slash 2 \Z$ and define the group $\mcG:= \{T^k \,|\, k \in \Z\} \simeq \Z$, where $T^k$ is the matrix in \eqref{TT}. Then the product group $\Z_2 \times \mcG$ acts on $\mcW$Polyg$(\R^2)$ as
$$ (\epsilon',T^k) \cdot (\De,b_{\ka},\epsilon) = (t_{u}(T^k(\De)),b_{\ka},\epsilon' \epsilon),$$ where as before $u = (\epsilon-\epsilon')/2$.}

\begin{de}
Let $(M,\om,(L,H))$ be a semitoric system with one focus-focus singularity, $b_\ka$ the vertical line through the corresponding critical value and $\epsilon$ a sign choice. Then the polygon invariant is the orbit of the $\Z_2 \times \mcG$ action
$$ (\Z_2 \times \mcG) \cdot (\De,b_{\ka},\epsilon) \in \mcW\text{Polyg}(\R^2)/(\Z_2 \times \mcG),$$ where $\De = f_\epsilon(B) \subset \R^2$ is a rational convex polygon and $f_\epsilon$ is a homeomorphism as in Theorem \ref{thmpol}.
\end{de}

In other words, the symplectic polygon invariant consists of a collection of $\Z_2 \times \Z$ weighted polygons, i.e., rational complex polygons together with the specification of a line $b_{\ka}$ and a sign choice $\epsilon \in \{-1,+1\}$.


\subsection{The definition of the height invariant} 

\label{sec:defheight}

Let $B:=F(M) \subseteq \R^2$ be the image of the momentum map and $f:B \to \R^2$ one of the possible homeomorphisms of Theorem \ref{thmpol}. The map $\mu:M \to \R^2$ defined by $\mu := f \circ F = f \circ (L,H)$ is called \emph{generalised toric momentum map} for the semitoric system $(M,\omega,(L,H))$, whose image $\De \subseteq \R^2$ is a rational convex polygon. The {\em height invariant} of the semitoric system $(M,\om,(L,H))$ is the number
$$ h := \mu(m) - \!\!\min_{s \in \De \cap b_{\ka}} \!\!\pi_2(s),$$ where $\pi_2:\R^2 \to \R$ is the canonical projection onto the second coordinate. It corresponds to the vertical height of $\mu(m_i)$ inside the polygon $\De$ and is independent of the choice of $f$ (Pelayo \& V\~u Ng\d{o}c \cite{PV1}).

The height invariant admits a more geometrical interpretation too. Let $Y = L^{-1}(c) \subset M$ and split it into two submanifolds: $Y^+ := Y \cap \{ p \in M \; : \; H(p) > H(m)\}$ and $Y^- := Y \cap \{ p \in M  : H(p) < H(m)\}$. Then $h$ is the symplectic volume of $Y^-$, i.e., the real volume divided by $2\pi$.


\subsection{The definition of the twisting-index invariant}

\label{sec:deftwistingIndex}
Consider a neighbourhood $W \subset M$ of the critical fibre. We take symplectic coordinates $(x_1,y_1,x_2,y_2)$ and consider the local diffeomorphism $\phi$ of $\R^2$ and the global map $\Phi=\phi \circ F$ as in \S \ref{sec:deftaylor}.  Since the Hamiltonian flow of $\Phi_2$ is $2\pi$-periodic, the second component $\Phi_2$ must coincide with $L$, possibly up to a sign. Applying the symplectic transformation $(x_1,y_1) \mapsto (-x_1,-y_1)$ on the definition of the local symplectic coordinates if needed, we can assume that on $W$ we have $\Phi_2 = L$, so $\phi$ is of the form $\phi(l,h) = (\phi_1(l,h),l)$.

Let us write $V:=F(W)$, so that $F^{-1}(V)$ is a neighbourhood of the singular fibre $F^{-1}(m)$, and consider the restriction of the map $\Phi$ to this neighbourhood, which for simplicity we will also denote by $\Phi$. Then close to any regular torus we have the $2\pi$-periodic Hamiltonian vector field
\begin{equation}
 2\pi \mcX_p := (\tau_1 \circ \Phi) \mcX_{\Phi_1} + (\tau_2 \circ \Phi) \mcX_L,
 \label{Xp}
\end{equation}
where $\tau_1, \tau_2$ are the functions defined in \eqref{taus}. This vector field is smooth on $F^{-1}(V \backslash b_{\ka})$.

On the one hand, there exists a unique smooth function $H_p:F^{-1}(V \backslash b_{\ka}) \to \R$ whose Hamiltonian vector field is $\mcX_p$ and satisfies $\lim_{x \to m} H_p =0$ (Lemma 5.6 of Pelayo \& V\~u Ng\d{o}c \cite{PV1}). The momentum map $\nu = (L,H_p)$ is called the {\em privileged momentum map} for $(L,H)$ around the focus-focus critical value $c$.

On the other hand, the generalised momentum map $\mu$ has the components $\mu=(\mu_1,\mu_2) = (L,\mu_2)$. The relation between $\mu$ and $\nu$ is given by
\begin{equation}
\mu = T^{k} \nu\quad \text{ on }F^{-1}(V),
\label{twistwis}
\end{equation}
where $k \in \Z$ is called the {\em twisting index of $\De_\text{weight}$ at the focus-focus critical value $c$}. If we apply a global transformation $T^r \in \mcG$, i.e., if we pick another representative of the polygon invariant, $\nu$ remains unchanged while $\mu$ transforms into $T^{r} \mu$. As a consequence, under such transformation all twisting indices change by $k \to k + r$.

Consider the space $W\text{Polyg}(\R^2) \times \Z$ of all weighted polygons of complexity one with their corresponding twisting indices. The group action of $\Z_2 \times \mcG$ on $W\text{Polyg}(\R^2) \times \Z$ is defined as follows:
$$ (\epsilon,T^{k'}) \star (\De,b_{\ka},\epsilon,k): = (t_{u}(T^k(\De)),b_{\ka},\epsilon' \epsilon,k+k').$$ 

This allows to define the twisting-index invariant. 

\begin{de}
	The {\em twisting-index invariant} of $(M,\omega,(L,H))$ is the $(\Z_2 \times \mcG)$-orbit of weighted polygons labelled by the twisting indices at the focus-focus singularities of the system given by
$$ (\Z_2 \times \mcG) \star (\De,b_{\ka},\epsilon,k) \in (W\text{Polyg}(\R^2) \times \Z) / (\Z_2 \times \mcG). $$	
\end{de}

\begin{re}
\label{twi1}
The twisting-index invariant is then completely determined by associating an integer index $k$ to one of the weighted polygons of the polygon invariant, or alternatively by finding a weighted polygon with index $k=0$, since then the associated integer index to the rest of the polygons can be reconstructed by knowing that $\Z_2$ does not act on the index and $\mcG \simeq \Z$ acts by addition.
\end{re}


\section{Local coordinates, elliptic integrals, and action}

\label{sec3}


\subsection{The spin-oscillator in various coordinates}

Let $\lambda,\mu>0$ be positive constants. Consider the product manifold $M={\mathbb S}^2 \times \mathbb{R}^2$ with symplectic form $\omega = \lambda\, \omega_{{\mathbb S}^2} \oplus \mu\, \omega_{\mathbb{R}^2}$, where $\omega_{{\mathbb S}^2}$ and $\omega_{\mathbb{R}^2}$ are the standard symplectic structures on the unit sphere and the Euclidean plane respectively. Let $(x,y,z)$ be Cartesian coordinates on the unit sphere ${\mathbb S}^2 \subset \mathbb{R}^3$ and $(u,v)$ Cartesian coordinates on the plane $\mathbb{R}^2$. A \emph{coupled spin-oscillator} is a 4-dimensional Hamiltonian integrable system $(M,\omega,(L,H))$, where the smooth map $F=(L, H): M \to \R^2$ is given by
\begin{equation}
L(x,y,z,u,v):= \mu \dfrac{u^2+v^2}{2} + \lambda (z-1) \quad \text{and} \quad H(x,y,z,u,v):=\dfrac{xu+yv}{2}.
\label{syst}
\end{equation}
Coupled spin-oscillators are completely integrable systems, i.e.\ the Poisson bracket $\{L,H\}$ vanishes and the system is of semitoric type, cf.\ Pelayo \& V\~u Ng\d{o}c \cite{PV3}. The only existing focus-focus singularity is at the point $m:=(0,0,1,0,0)$, so we have $\nff=1$. We have shifted the value of $L$ by $\lambda$ so that we have $(L,H)(m)=(0,0)$. 

The system \eqref{syst} induces a foliation $\mcF$ on $M$. This foliation has some discrete symmetries. For example, the transformations
\begin{align}
T_1:&\; \quad x \mapsto -x,\; \quad  y\mapsto -y,\;  \quad H \mapsto -H\nonumber\\
T_2:&\; \quad u \mapsto -u,\; \quad  v \mapsto -v,\; \quad  H \mapsto -H
\label{transf}
\end{align} leave the symplectic form $\om$ and the foliation $\mcF$ unchanged. 
The group $\mathbb{Z}_2 \times \mathbb{Z}_2$ acts on $M$ thus by symplectic transformations. Since the Taylor series invariant is an invariant of the foliation $\mcF$, it must remain unchanged by these transformations.

In order to reduce the system by the ${\mathbb S}^1$-action, we define symplectic coordinates so that we can express the functions $L,H$ in a simpler way:
\begin{equation*}
\begin{array}{lll}
z:= \pm \sqrt{1-x^2-y^2}, && \rho := \dfrac{u^2+v^2}{2}, \\[0.3cm]
\theta := \arg \left( x+iy \right), && \varphi := \arg \left( u+iv\right). \\
\end{array}
\end{equation*} In these coordinates the functions $L$ and $H$ become
\begin{equation*}
L(z,\theta,\rho,\varphi)= \mu \, \rho + \lambda (z-1) \quad \text{and} \quad H(z,\theta,\rho,\varphi)=\dfrac{\sqrt{\rho (1-z^2)}\cos (\theta - \varphi)}{\sqrt{2}}.
\end{equation*}
The symplectic form in these coordinates is $\omega = \lambda\, d z \wedge d \theta + \mu\, d \rho \wedge d \varphi$. We now perform a linear coordinate change, given by
\begin{equation*}
\begin{array}{lll}
q_1:= \theta, && q_2 := \varphi-\theta, \\[0.3cm]
p_1 := \mu\, \rho + \lambda (z-1), && p_2 := \mu\, \rho. \\
\end{array}
\end{equation*} In these coordinates, $\omega$ becomes the standard symplectic form $\omega = d p_1 \wedge d q_1 + d p_2 \wedge d q_2$ and the functions become $L(q_1,p_1,q_2,p_2)= p_1$ and 
\begin{equation}
H(q_1,p_1,q_2,p_2)=\sqrt{\dfrac{-p_2(p_2-p_1)(p_2-p_1-2\lambda)}{2 \lambda^2 \mu}}\cos (q_2).
\label{eq1}
\end{equation} It is important to see that the Hamiltonian function $H$ is independent of $q_1$, which means that $p_1$ is a constant 
along the flow of $H$. More precisely, $L=p_1$ is the total angular momentum of the system and generates the global ${\mathbb S}^1$-action corresponding to simultaneous rotation of the sphere about the vertical axes and of the plane about the origin. From the coordinate expressions we see that $p_1 \geq -2\lambda$ and $p_2$ must satisfy the inequalities $p_2 \geq 0$, $p_2 \geq p_1$ and $p_2 \leq p_1 + 2\lambda$. Under these circumstances, the argument of the square root is non-negative and  $H$ is well-defined.


\subsection{Elliptic integrals}

\label{sec:ellipticInt}

Elliptic integrals appear often in different areas of mathematics and physics, such as classical mechanics or complex function theory. In this section we review some basic definitions and properties of elliptic integrals that we use throughout the paper. For a more detailed discussion, see for example Siegel \cite{Si1}, \cite{Si2} and Bliss \cite{Bl}.

Let $z,w$ be two variables, which might either be real or complex,  $R(z,w)$ be a rational function and $P(z)$ a polynomial of degree three or four. Integrals of the form 
\begin{equation}
\mathcal{N}(x) := \int_c^x R(z,\sqrt{P(z)})\, d z,
\label{ellip}
\end{equation} where $c$ is a constant, are called \emph{elliptic integrals}. Except in special situations, such as $R(z,w)$ depending only on even powers of $w$ or $P(z)$ having repeated roots,  elliptic integrals cannot be expressed in terms of elementary functions. It is however possible to express them in terms of integrals of rational functions and the three \emph{Legendre canonical forms}:
\begin{align*}
& F(x; k)  := \int_0^x \dfrac{dt}{\sqrt{(1-t^2)(1-k^2t^2)}}, \\
& E(x; k)  := \int_0^x \dfrac{\sqrt{1-k^2t^2}}{\sqrt{1-t^2}}dt, \\
& \Pi(n; x; k)  := \int_0^x \dfrac{dt}{(1-nt^2) \sqrt{(1-t^2)(1-k^2t^2)}}.
\end{align*} 
The functions $F(x; k)$, $E(x; k)$ and $\Pi(n; x; k)$ are called \emph{incomplete integral of first, second} and \emph{third kind} respectively. The number $k$ is called the \emph{(elliptic) modulus} or \emph{excentricity}, $n$ is said to be the \emph{characteristic} and $x$ sometimes simply receives the name of \emph{argument}.

$F(x;k)$ is finite for all real or complex values of $x$, including infinity. $E(x;k)$  has a simple pole of order one at $x=\infty$ and $\Pi(n; x; k)$ is logarithmically infinite for $x^2=\frac{1}{n}$. For the particular case $x=1$, the integrals are said to be \emph{complete}. More precisely, $K(k):=F(1;k)$ is the \emph{complete elliptic integral of first kind}, $E(k):=E(1;k)$ is the \emph{complete elliptic integral of second kind} and $\Pi(n,k) := \Pi(n; 1; k)$ is the \emph{complete elliptic integral of third kind}.

For most applications in classical mechanics, the variables $z$ and $w$ of \eqref{ellip} are real. However, it is often convenient to allow them to take complex values. In this case, the elliptic curve 
\begin{equation*}
\Gamma := \{ (z,w) \in \mathbb{C}^2 \; :\; w^2 = P(z) \} 
\end{equation*} will be a \emph{Riemann surface}, i.e.\ a one-dimensional complex manifold. Using complex function theory it can be shown that for the elliptic case $\deg P = 3$ or $4$ the surface $\Gamma$ is homeomorphic to a torus. More generally, if $n=\deg P$, then $\Gamma$ will be a compact surface of genus $\frac{n-2}{2}$ if $n$ is even or genus $\frac{n-1}{2}$ if $n$ is odd, assuming that the curve is non-singular.

Restricting ourselves to the elliptic case, $\Gamma$ is homeomorphic to a torus $\mathbb{T}^2 = \mathbb{S}^1 \times \mathbb{S}^1$. This means that the rank of its first homology group is 2, generated by two independent non-contractible cycles $\alpha$ and $\beta$ (see Figure \ref{cycles}). When computing the integral of the action on the regular fibres close to the singular fibre containing the focus-focus point, we will refer to the  cycle $\alpha$ as the {\em `imaginary'} or {\em `vanishing' cycle}, since it becomes arbitrarily small as we approach the singular fibre, and $\beta$ as the {\em `real' cycle}, since it corresponds to the real elliptic curve obtained by considering the values of the variables $z,w$ of $\Ga$ to be real.

\begin{figure}[ht]
 \centering
 \includegraphics[width=6cm]{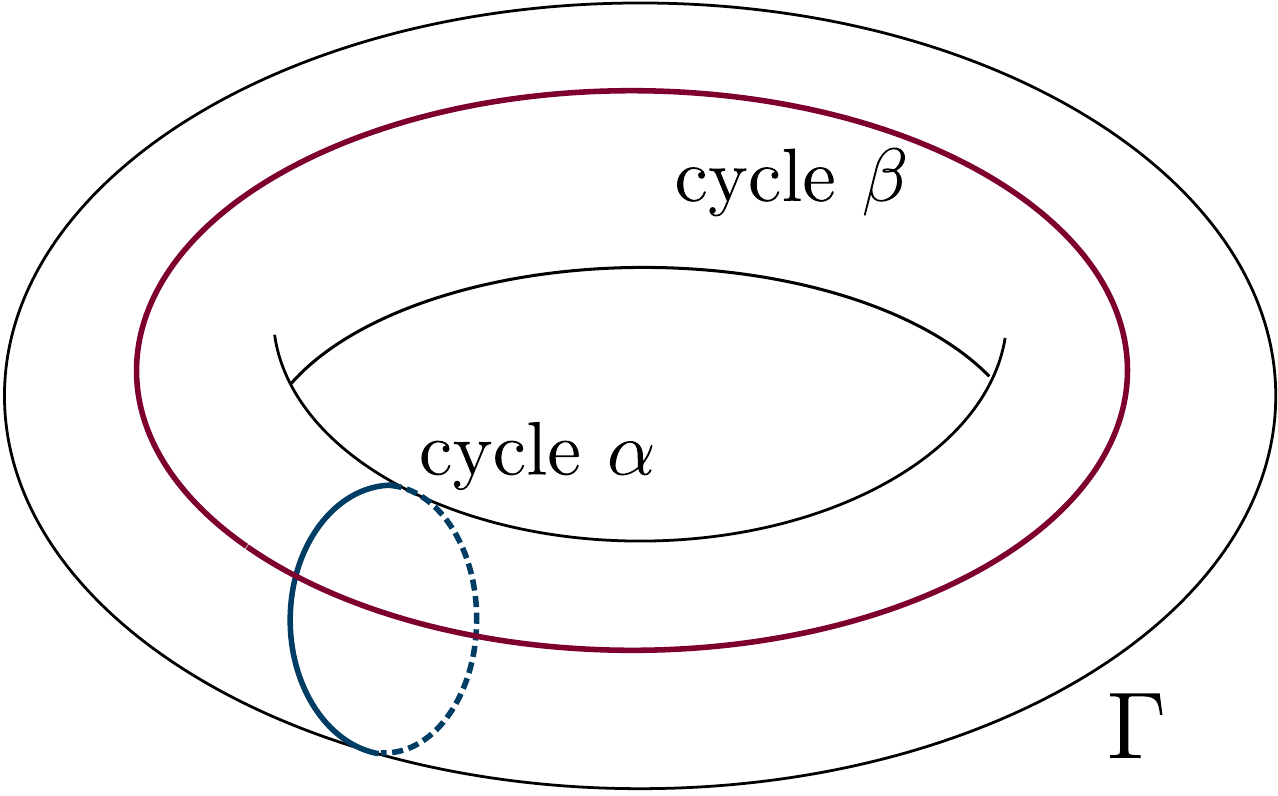}
 \caption{\small The elliptic curve $\Gamma$ with the imaginary cycle $\alpha$ (blue colour) and the real cycle $\beta$ (red colour).}
 \label{cycles}
\end{figure}


\subsection{The action integral}

Let $l$ and $h$ be the values of the functions $L$ and $H$, which are conserved quantities of the system. Since $p_1=l$ is a first integral of the system, we can perform symplectic reduction on the level $L=l$ for any $l \geq -2\lam$ and consider the reduced system
\begin{equation}
H_{red}(q_2,p_2):=\sqrt{\dfrac{-p_2(p_2-l)(p_2-l-2\lambda)}{2 \lambda^2 \mu}}\cos (q_2),
\label{eqred}
\end{equation}  
depending only on coordinates $(q_2,p_2)$. Note that the function $H_{red}(q_2,p_2)$ is even with respect to $q_2$, i.e.\ $H_{red}(-q_2,p_2)=H_{red}(q_2,p_2)$. The phase space is the finite cylinder given by 
\begin{equation}
\max\{0,l\} \leq p_2 \leq l+2\lam \quad \mbox{and} \quad -\pi\leq q_2 \leq \pi.
\label{vol}
\end{equation}
This reduced system is defined on an orbifold that is obtained by symplectic reduction of the symmetry induced 
by the global $\mbS^1$-action. The coordinates $(q_2, p_2)$ are a system of global singular coordinates on the orbifold. For $l \not = 0$ 
the reduced phase space is in fact a smooth manifold (or a point), while for $l=0$ it has a single conical singularity. This can best
be seen by considering the Hilbert invariants of the action of the flow of $L$, which are
$\rho_1 = u^2 + v^2$, $\rho_2 = x^2 + y^2$, $\rho_3 =  u x + v y$, $\rho_4 = u y - v x$ and $z$,
with relation $\rho_3^2 + \rho_4^2 = \rho_1 \rho_2$.
It is easy to check that
their Poisson bracket closes, and has two Casimirs $C_1 = \frac12 \mu \rho_1 + \lambda( z-1)  - l = 0$ and $C_2 = \rho_2 + z^2 - 1= 0$. 
Eliminating $\rho_1$ and $\rho_2$ using these Casimirs gives a Poisson structure on $\R^3$ with coordinates 
$(\rho_3, \rho_4, z)$ with Casimir
$C_3 = \rho_3^2 + \rho_4^2 - \tfrac{2}{\mu} ( l -\lambda( z-1)) ( 1- z^2)  = 0$. 
The zero-level of $C_3$ defines the reduced phase space.

To see if and when the reduced space is singular we need to find points
on $\{ (\rho_3, \rho_4, z) : C_3 = 0 \} $ at which the gradient of $C_3$ with respect to $(\rho_3, \rho_4, z)$ vanishes.
This occurs when $\rho_3 = \rho_4 = 0$ and the conditions $C_3 = 0$ and 
$\frac{\partial C_3}{\partial z} = 0$ together then imply $l= 0$ or $l = - 2 \lambda$.
For $l=-2\lambda$ the reduced space is a point at $z=-1$, 
while for $l=0$ the reduced space is an orbifold  in the shape of a balloon with a singular point at $z=1$.
Since $H = \frac{1}{2}\rho_3$ the reduced dynamics is $\dot \rho_3 = 0$ and hence the reduced orbit is given by 
the intersection of the plane $\rho_3 = 2 H = \mathrm{const}$ with the balloon. Thus exactly when $h=0$ and $l=0$ the
reduced orbit contains the conical singularity.
Introducing  $p_2 = l - \lambda(z - 1)$ and $q_2 = \arctan( \tfrac{\rho_3}{\rho_4})$ defines
local coordinates almost everywhere on the orbifold $C_3 = 0$. The final expression for the action $I$ derived 
in the following measures area on this orbifold.

By the Liouville-Arnold theorem, the system \eqref{eq1} has two pairs of action-angle variables near a torus in the preimage of a regular value.
One of the global action variables is $p_1 = l$ and we will denote the other non-trivial action variable by $I$.
Action variables are invariant under symplectic transformations and hence invariant under symplectic reduction obtained by a symplectic 
transformation, as in our case. Thus the action variable of the reduced system \eqref{eqred} gives the non-trivial action $I$ of the original system. Computing action integrals is in principle straightforward, but we always hope that the 
resulting integrals are (complete) Abelian integrals. Using the standard expression $\oint p\,\dee q$ 
to compute the action in our case would require to solve a cubic equation.
The simple trick to instead integrate $\oint q\, \dee p$ is obvious, but does not lead to an
Abelian integral. Similar to the case of the Kovalevskaya top, cf.\ Dullin \& Richter \& Veselov \cite{DRV}, by using integration by 
parts the integral can be turned into an Abelian integral, in the present case in fact a complete elliptic integral.
We start by defining
\begin{equation}
I(l,h) := \dfrac{1}{2\pi} \oint_{\beta_{l,h}} q_2\, \dee p_2,
\label{acti}
\end{equation} where $\beta_{l,h}$ is the curve implicitly defined by $H_{red}(q_2,p_2)=h$. For later convenience we choose the orientation of $\beta_{l,h}$ such that $\frac{\partial I}{\partial h}>0$. In \eqref{acti},  $q_2 = q_2(p_2; l,h)$ is thought as a function of $p_2$ depending on the parameters $l$ and $h$, found from \eqref{eqred}
\begin{equation}
q_2(p_2; l,h) = \pm \arccos \left(  \dfrac{ \lambda\sqrt{2 \mu} h}{\sqrt{-p_2(p_2-l)(p_2-l-2\lambda)}} \right). 
\label{qq}
\end{equation} Since $H_{red}(q_2,p_2)=H_{red}(-q_2,p_2)$ we can restrict the integral to the positive values of $q_2$ and then multiply by 2. In Figure \ref{levels} we can see a figure of the level curves of $H_{red}$. 

\begin{figure}[ht]
 \centering
 \includegraphics[width=12cm]{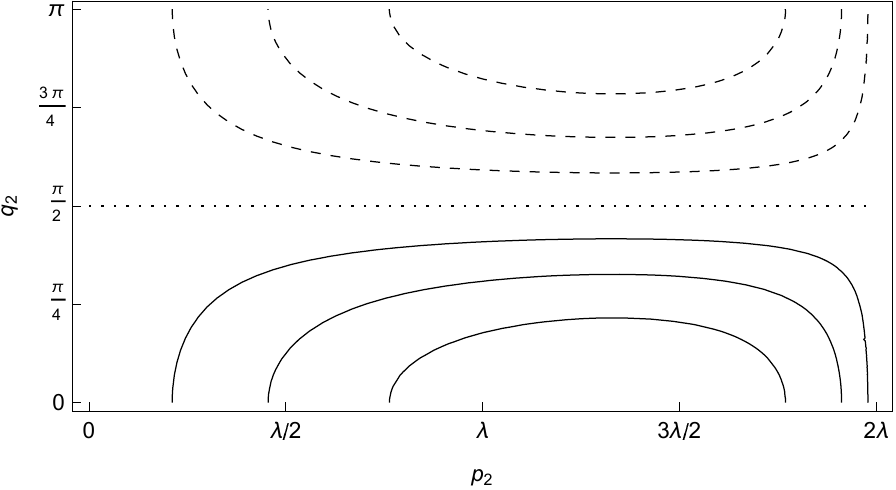}
\caption{\small Representation of some level sets of the function $H_{red}(q_2,p_2)$ with parameters $\lam=1$, $\mu=1$, $l=0$. The continuous lines correspond to positive values of the function and the dashed lines to negative ones. The dotted line is the level set of value $0$.}
 \label{levels}
\end{figure}

\begin{re} The function $H_{red}(q_2,p_2)$ presents also the symmetry
\begin{equation}
H_{red}(q_2+\pi,p_2) = H_{red}(q_2-\pi,p_2) = -H_{red}(q_2,p_2)
\label{tf1}
\end{equation}  which combined with the symmetry with respect to $q_2=0$ leads to
$$ H_{red}(\pi-q_2,p_2) = H_{red}(q_2-\pi,p_2) = -H_{red}(q_2,p_2).$$ In other words, $H_{red}(q_2,p_2)$ is antisymmetric with respect to the line $q_2 = \frac{\pi}{2}$. This symmetry is visible in Figure \ref{levels}. If $I$ is the action corresponding to a certain value $h$ of $H_{red}$, then the action $I'$ corresponding to $-h$ will be
\begin{equation}
I' = \frac{ \text{Area}_l }{2\pi} - I,
\label{I}
\end{equation} where $\text{Area}_l := 2\pi( 2\lam+\min\{0,l\})$ is the area of the phase space \eqref{vol}. For this reason from now on we will assume without loss of generality that $h>0$.
\label{symmet}
\end{re}

\begin{re}
The transformations $T_1, T_2$ defined in \eqref{transf} act on the angles $\theta$, $\varphi$ respectively by substracting $\pi$ if the angles are positive or adding $\pi$ if they are negative, so as to always have angular values in $[-\pi,\pi]$. In a similar fashion, they both act on $q_2$ the same way: they substract or add $\pi$ if $q_2$ is positive or negative respectively. From \eqref{tf1} we see that this changes also the sign of $H_{red}$. Following the previous remark,  $T_1$ and $T_2$ take $I$ to $I'$ and vice versa. \label{symmet2}
\end{re}

The integral \eqref{acti} can be expressed as an integral  of elliptic type and therefore can be calculated in terms of Legendre's standard elliptic integrals. We will use the notations, definitions and results concerning elliptic integrals from section \ref{sec:ellipticInt}.

\begin{theo}
The action integral of the reduced spin-oscillator is given by 
\begin{equation*}
2\pi I = c_1 K(k)+ c_2 \Pi(n_2,k) + c_3 \Pi (n_3,k),
\end{equation*} where $K$ and $\Pi$ are Legendre's complete elliptic integrals of first and third kind respectively. The elliptic modulus and the characteristics are
\begin{equation*}
k^2=\dfrac{\zeta_3-\zeta_2}{\zeta_3-\zeta_1},\qquad n_2 =\dfrac{\zeta_3-\zeta_2}{\zeta_3-l},\qquad n_3 =\dfrac{\zeta_3-\zeta_2}{\zeta_3-l-2\lambda},
\end{equation*} the coefficients are 
\begin{equation*}
c_1 = \dfrac{3\lambda \sqrt{2\mu}\,h}{\sqrt{\zeta_3-\zeta_1}},\quad c_2 = \dfrac{\lambda\sqrt{2\mu}\,h\, l }{\sqrt{\zeta_3-\zeta_1}} \dfrac{1}{\zeta_3-l},\quad c_3= \dfrac{\lambda \sqrt{2 \mu}\,h\,  (l+2\lambda) }{\sqrt{\zeta_3-\zeta_1}} \dfrac{1}{\zeta_3-l-2\lambda}
\end{equation*}
and $\zeta_1 \leq \zeta_2 \leq \zeta_3 $ are the roots of 
\begin{equation*}
P(p_2) =  -\frac{2}{\lambda^2 \mu}p_2 (p_2-l) (p_2-l-2\lambda)-4h^2.
\end{equation*}
\label{theo2}
\end{theo}

In order to make the computations easier to follow, we will make use of the following notation. 

\begin{no}
We introduce the following notation for scaled variables
\begin{equation}
\mfp_1 := \dfrac{1}{\lambda}p_1, \qquad \mfp_2 := \dfrac{1}{\lambda}p_2,
 \label{not1}
\end{equation} 
together with the corresponding scaled functions
\begin{equation}
\mfL := \dfrac{1}{\lambda}L, \qquad \mfH := \sqrt{\dfrac{\mu}{\lambda}}H,\qquad \mfI := \dfrac{1}{\lambda}I
 \label{not2}
\end{equation} 
and function values
\begin{equation}
\mfl := \dfrac{1}{\lambda}l, \qquad \mfh := \sqrt{\dfrac{\mu}{\lambda}}h.
 \label{not3}
\end{equation}
\end{no}

\begin{proof}[Proof of Theorem \ref{theo2}]

In the scaled notation \eqref{not1}-\eqref{not3} the integral \eqref{acti} becomes
\begin{equation}
\mfI(\mfl,\mfh) =\dfrac{1}{2\pi} \oint_{\beta_{\mfl,\mfh}} \arccos \left( \dfrac{\sqrt{2}\, \mfh }{\sqrt{-\mfp_2 (\mfp_2-\mfl) (\mfp_2-\mfl-2)}} \right)d\mfp_2,
\label{act1}
\end{equation} where $\beta_{\mfl,\mfh}$ is the curve implicitly defined by $\mfH(q_2,\mfp_2)=\mfh$.

Integrating by parts using $\frac{d}{dt} \arccos(t) = - \frac{1}{\sqrt{1-t^2}}$, we can rewrite \eqref{act1} as
\begin{align}
\mfI(\mfl,\mfh)&=\dfrac{1}{\pi} \int_{\zeti_2}^{\zeti_3} \dfrac{\mfh({\mfl}^2-4\mfl\mfp_2+3{\mfp_2}^2+2\mfl  -4\mfp_2 )}{\sqrt{2} (\mfl-\mfp_2)(\mfl-\mfp_2+2)} \dfrac{d\mfp_2}{\mfw} \nonumber \\[0.2cm]
&= \dfrac{\mfh}{\pi}    \int_{\zeti_2}^{\zeti_3} \left( 3 +  \dfrac{ \mfl}{(\mfp_2-\mfl)}   + \dfrac{\mfl+2}{(\mfp_2-\mfl-2)}     \right) \dfrac{d\mfp_2}{\mfw} \,.
\label{act4}
\end{align} 
This integral is  an elliptic integral defined on the elliptic curve
\begin{equation}
\Gamma_{\mfl,\mfh}:=\{ (\mfp_2,\mfw): \mfw^2=\mfP(\mfp_2) \}, \qquad \mfP(\mfp_2) := -2\mfp_2 (\mfp_2-\mfl) (\mfp_2-\mfl-2)-4\mfh^2.
\label{curv}
\end{equation}
The roots of  $\mfP$ are denoted by $\zeti_i$, $i=1,\ldots,3$ and they satisfy $\zeti_1 \leq \min\{0,l\}$ and $\max\{0,\mfl\} \leq \zeti_2 \leq \zeti_3 \leq \mfl+2$. The integration takes place  between the two roots that lie in the physical phase space, namely $\zeti_2$ and $\zeti_3$. 

We observe that the last two terms in \eqref{act4} have poles and lead to elliptic integrals of third kind. Their residues, which in this case coincide with the position of the poles, are actions scaled by the global factor $\frac{1}{2\pi}$. The first term corresponds to the $\mathbb{S}^1$-action $l$, and the $2$ in the second term, or $2\lambda$ if we revert the scaling, corresponds to the symplectic area of the sphere $\mathbb{S}^2$.

We can decompose \eqref{act4} into an integral of the form
\begin{equation*}
\mathcal{N}_A := \int_{\zeti_2}^{\zeti_3} \dfrac{d\mfp_2}{\mfw} = \int_{\zeti_2}^{\zeti_3} \dfrac{d\mfp_2}{\sqrt{\mfP(\mfp_2)}}
\end{equation*}  
and two integrals of the form 
\begin{equation*}
\mathcal{N}_{B,\gamma} := \int_{\zeti_2}^{\zeti_3} \dfrac{1}{\mfp_2-\gamma}\dfrac{d\mfp_2}{\mfw} = \int_{\zeti_2}^{\zeti_3} \dfrac{d\mfp_2}{(\mfp_2-\gamma)\sqrt{\mfP(\mfp_2)}},
\end{equation*} where $\gamma$ is a constant. We can rewrite $\mcN_A,\mcN_{B,\gamma}$ in terms of Legendre's standard form by performing a change of integration variable 

\begin{equation}
x := \sqrt{\dfrac{\zeti_3-\mfp_2}{\zeti_3-\zeti_2}}
\end{equation} and defining
\begin{equation}
\mfk^2:=\dfrac{\zeti_3-\zeti_2}{\zeti_3-\zeti_1},\qquad \mfn_\gamma :=\dfrac{\zeti_3-\zeti_2}{\zeti_3-\gamma}.
\end{equation} We obtain $\mathcal{N}_A$ as a complete elliptic integral of first kind
\begin{equation}
\mathcal{N}_A = \dfrac{\sqrt{2} }{\sqrt{\zeti_3-\zeti_1}} \int_0^1 \dfrac{dx}{\sqrt{(1-x^2)(1-\mfk^2 x^2)}} = \dfrac{\sqrt{2}  }{\sqrt{\zeti_3-\zeti_1}} K(\mfk)
\label{Ia}
\end{equation} and $\mathcal{N}_{B,\gamma}$ as a complete integral of third kind
\begin{align}
\mathcal{N}_{B,\gamma} &= \dfrac{\sqrt{2} }{(\zeti_3-\gamma)\sqrt{\zeti_3-\zeti_1}} \int_0^1 \dfrac{dx}{(1-\mfn_\gamma x^2)\sqrt{(1-x^2)(1-\mfk^2 x^2)}}
\nonumber \\  &= \dfrac{\sqrt{2} }{(\zeti_3-\gamma)\sqrt{\zeti_3-\zeti_1}} \Pi(\mfn_\gamma,\mfk).
\label{Ib}
\end{align} We combine now \eqref{act4} with \eqref{Ia} and \eqref{Ib} and transform back to the unscaled variables $p_2, l,h$. This way we obtain the desired result. 
\end{proof}

We now want to understand the behaviour of the elliptic integral $I(l,h)$ around the focus-focus critical value $(0,0)$. The integral $\mfI(\mfl,\mfh)$ in \eqref{act4} can either be understood as a real elliptic integral or as the real $\beta$ cycle of a complex elliptic integral. In unscaled coordinates we can thus write
\begin{equation*}
I(l,h) = \dfrac{1}{2\pi} \oint_{\beta_{l,h}} -h \left( 3 +  \dfrac{ l}{(p_2-l)}   + \dfrac{l+2\lam}{(p_2-l-2\lam)}     \right) \dfrac{dp_2}{w},
\end{equation*} defined on
\begin{equation*}
\Gamma_{l,h}:=\{ (p_2,w): w^2=P(p_2) \}, \qquad P(p_2) =  -\frac{2}{\lambda^2 \mu}p_2 (p_2-l) (p_2-l-2\lambda)-4h^2.
\end{equation*}It is known from complex analysis that when expanding an elliptic integral as a series, the integral along the imaginary or vanising $\alpha$ cycle appears in front of a logarithmic term. To this end, it is convenient to define the {\em `imaginary action'}:  the elliptic integral of the action but this time integrated along the imaginary $\alpha$ cycle. In unscaled coordinates it takes the form
\begin{equation}
J(l,h) := \dfrac{1}{2\pi i} \oint_{\alpha_{l,h}} h \left( 3 +  \dfrac{ l}{(p_2-l)}   + \dfrac{l+2\lam}{(p_2-l-2\lam)}     \right) \dfrac{dp_2}{w} \,.
\label{eqJ}
\end{equation}where $i$ is the imaginary unit. We denote by $\mfJ:=\frac{1}{\lambda}J$ its scaled counterpart. 
Here the $\alpha$ cycle needs to be defined using a cycle in the original coordinates, 
and because of the transformation $p_2 = u^2 + v^2$ the original cycle is wrapped around twice. 
\begin{lemm}
The series expansion of the action integral $I(l,h)$ of the spin-oscillator as a function of the angular momentum value $l$ and the energy value $h$ is
\begin{align}
2\pi &I(l,h)= 2\lambda\pi   +\dfrac{\pi}{2} l + l \arctan \left( \dfrac{l}{2\sqrt{\lam \mu}\,h} \right) \nonumber + J(l,h) \log \dfrac{32\lambda}{\sqrt{{l}^2+4\lam\mu  h^2}} \nonumber\\
&+ 2\sqrt{\lam\mu}h +	\dfrac{\sqrt{\lam\mu}}{2\lambda} lh  - \dfrac{\sqrt{\lam\mu}}{384\lambda^2({l}^2+4\lam\mu\,  h^2)} h
\left(63l^4+412\lam\mu l^2h^2 +544\lam^2\mu^2 h^4 \right)  \label{act2}\\
&+\dfrac{\sqrt{\lam\mu}}{3072\lambda^3({l}^2+4\lam\mu\,  h^2)^2}lh(185l^6 + 2668 \lam \mu l^4 h^2 + 12176 \lam^2 \mu^2 l^2h^4+18112 \lam^3\mu^3 h^6)+...\nonumber
\end{align} 
where the scaled imaginary action $J(l,h)$ has the expansion
\begin{align}
\dfrac{1}{\sqrt{\lam \mu}} J(l,h) &= 2h - \dfrac{1}{4\lambda} l h  +\dfrac{1}{128\lambda^2 } h (9{l}^2+20\lam\mu h^2) - \dfrac{5}{1024\lambda^3}lh \left(5l^2+28\lam\mu h^2  \right) \nonumber \\
&+\dfrac{7}{131072\lambda^4}h \left( 175l^4+1800\lam\mu l^2h^2+1584\lam^2\mu^2 h^4 \right)  +...
\label{expJ}
\end{align}
\label{act3}
\end{lemm}

\begin{proof}

The procedure is similar to the one used by Dullin \cite{Du}. The idea is to expand Legendre's elliptic integrals in series in the singular limit of the modulus $\mfk\to1$ and then substitute the modulus and the parameters with their series expansions around the focus-focus critical value $(\mfl,\mfh)=(0,0)$. To do so, we set  $\mfl \mapsto \mfl \varepsilon$ and $\mfh \mapsto \mfh\varepsilon$  and we compute the Taylor expansion in $\varepsilon$. We start with the roots of $\mfP(\mfp_2)$:
\begin{align*}
\zeti_1 &= \dfrac{1}{2} \left( \mfl - \sqrt{{\mfl}^2+4 \mfh^2} \right)\eps + \dfrac{1}{4} \mfh^2 \left( 1+ \dfrac{\mfl}{\sqrt{{\mfl}^2+4 \mfh^2}} \right)\eps^2 + \mathcal{O}(\varepsilon^3), \\
\zeti_2 &= \dfrac{1}{2} \left( \mfl + \sqrt{{\mfl}^2+4 \mfh^2} \right)\eps + \dfrac{1}{4} \mfh^2 \left( 1- \dfrac{\mfl}{\sqrt{{\mfl}^2+4 \mfh^2}} \right)\eps^2 + \mathcal{O}(\varepsilon^3), \\
\zeti_3 &= 2+ \mfl\eps - \dfrac{1}{2} \mfh^2\eps^2 + \mathcal{O}(\varepsilon^3)
\end{align*} and from here we can calculate the modulus and the parameters
\begin{align*}
\mfk^2 &= 1-\dfrac{\sqrt{{\mfl}^2+4 h^2}}{2}\eps + \dfrac{1}{8} \left({\mfl}^2+4 \mfh^2+ \dfrac{{\mfl}^3+6 \mfh^2 \mfl}{\sqrt{{\mfl}^2+4 \mfh^2}} \right)\eps^2 + \mathcal{O}(\varepsilon^3), \\
\mfn_2 &= 1+ \dfrac{\mfl - \sqrt{{\mfl}^2+4 \mfh^2}}{4}\eps - \dfrac{1}{8 } \left( 1- \dfrac{\mfl}{\sqrt{{\mfl}^2+4 \mfh^2}} \right) \eps^2 + \mathcal{O}(\varepsilon^3), \\
\mfn_3 &=- \dfrac{4 }{ \mfh^2\eps^2} + \dfrac{-3\mfl +\sqrt{{\mfl}^2+4\mfh^2} }{ \mfh^2\eps} + \mathcal{O}(\varepsilon^0).
\end{align*}

Both $\mfn_2$ and $\mfn_3$ belong to the so-called \emph{circular case}, i.e., when either $k^2 < n < 1$ or $n<0$, cf.\ Cayley \cite{Ca}. This means that we can rewrite the complete elliptic integral of third kind in terms of Heuman's lambda function $\Lambda_0$ defined by
\begin{equation}
\Lambda_0(\vartheta,k) := \dfrac{2}{\pi} \left( E(k) F(\vartheta,k') + K(k)E(\vartheta,k')-K(k)F(\vartheta,k') \right),
\end{equation} 
where $E$ is the elliptic integral of second kind, $F$ is the incomplete elliptic integral of first kind and $k' = \sqrt{1-k^2}$ (see for more details Abramowitz \& Stegun \cite{Ab1} and Byrd \& Friedman \cite{BF}). More precisely, $\mfn_2$ belongs to the positive circular case ($k^2 < n < 1$), where  
\begin{equation*}
\Pi(n,k) = K(k) + \dfrac{\pi}{2} \sqrt{\dfrac{n}{(1-n)(n-k)}} \left( 1-\Lambda_0(\vartheta,k) \right),\quad \vartheta = \arcsin \sqrt{\dfrac{1-n}{n-k}} 
\end{equation*} 
and $\mfn_3$ belongs to the negative circular case ($n<0$), where
\begin{equation*}
\Pi(n,k) = \dfrac{1}{1-n}K(k) + \dfrac{\pi}{2} \sqrt{\dfrac{n}{(1-n)(n-k)}} \left( 1-\Lambda_0(\vartheta,k) \right),\quad \vartheta = \arcsin \dfrac{1}{\sqrt{-n}}.
\end{equation*} 
Finally we expand the functions
\begin{align*}
K(k) &= \dfrac{1}{4}(k^2-1)- \dfrac{21}{128} (k^2-1)^2 + \dfrac{185}{1536} (k^2-1)^3 + ... \\
 &+ \left( -\dfrac{1}{2} +\dfrac{1}{8}(k^2-1) - \dfrac{9}{128} (k^2-1)^2 - \dfrac{25}{512}(k^2-1)^3 + ... \right) \log \left( \dfrac{1-k^2}{16} \right)\\
 \Lambda_0(\vartheta,k) &= \dfrac{2}{\pi} \vartheta + \left(  \dfrac{1}{2\pi}(k^2-1) - \left( \dfrac{13}{32\pi} + \dfrac{3\sin^2 \vartheta}{16\pi} \right)(k^2-1)^2  + ...\right) \sin \vartheta \cos \vartheta \\
 &+ \left(  \dfrac{1}{2\pi}(k^2-1) - \left( \dfrac{3}{16\pi} + \dfrac{\sin^2 \vartheta}{8\pi} \right)(k^2-1)^2  + ...\right) \sin \vartheta \cos \vartheta \log \left( \dfrac{1-k^2}{16} \right).
\end{align*} Substituting the modulus and the angles we eventually obtain
\begin{align}
2\pi &\mfI(\mfl,\mfh)= 2\pi   +\dfrac{\pi}{2} \mfl + \mfl \arctan \left( \dfrac{\mfl}{2\mfh} \right) \nonumber + \mfJ(\mfl,\mfh) \log \dfrac{32}{\sqrt{{\mfl}^2+4 \mfh^2}} \nonumber\\
&+ 2\mfh +	\dfrac{1}{2} \mfl\mfh  - \dfrac{1}{384({\mfl}^2+4  \mfh^2)} \mfh
\left(63\mfl^4+412\mfl^2\mfh^2 +544 \mfh^4 \right) \\
&+\dfrac{1}{3072({\mfl}^2+4 \mfh^2)^2}\mfl\mfh(185\mfl^6 + 2668  \mfl^4 \mfh^2 + 12176  \mfl^2\mfh^4+18112 \mfh^6)+...\nonumber
\end{align} 
where the scaled imaginary action $\mfJ(\mfl,\mfh)$ has the expansion
\begin{align}
\mfJ(\mfl,\mfh) &= 2\mfh - \dfrac{1}{4} \mfl \mfh  +\dfrac{1}{128 } \mfh (9{\mfl}^2+20 \mfh^2) - \dfrac{5}{1024}\mfl\mfh \left(5\mfl^2+28\mfh^2  \right) \nonumber \\
&+\dfrac{7}{131072}\mfh \left( 175\mfl^4+1800 \mfl^2\mfh^2+1584\mfh^4 \right)  +...
\label{expJ2}
\end{align}
Reverting the scaling of variables we obtain \eqref{act2}.
\end{proof}


\section{Calculation of the Taylor series invariant}

We want to extract the symplectic invariant from the expression \eqref{inv5}. In the previous section we have computed the action as a function of $h$ and $l$. In order to eliminate $h$ the Birkhoff normal form 
needs to be computed, which allows to express $h$ as a function of $j$ and $l$. Here $j$ and $l$ are the semi-global extensions of $Q_1$ and $Q_2$, respectively, as defined in  \S \ref{sec:deftaylor}.

The observation that the Birkhoff normal form can be computed by inversion of an elliptic integral was first made in Dullin \cite{Du}. This integral is defined on the same curve $\Gamma$, but computed along a different cycle. In fact it is the imaginary action $J(l,h)$ that we have already defined. Here we are going to compute the imaginary action $J(l,h)$ by direct expansion and using residue calculus. Combining this with Lemma \ref{act3}, we will obtain the desired result. 

\begin{lemm}
\label{lact}
The  Birkhoff normal form of the focus-focus point of the spin-oscillator is 
\begin{align}
\sqrt{\lam\mu}\,B(j,l)&=\dfrac{1 }{2 }j +\dfrac{ 1}{16\lambda}{j}{l} -\dfrac{5}{512\lambda^2}j\left({j}^2+{l}^2\right)+\dfrac{ 1}{4096\lambda^3 } j\left(15 {j}^2{l}+11  {l}^3\right)\nonumber \\
& -\dfrac{ 1}{524288\lambda^4}j\left(393 {j}^4+990 {j}^2 {l}^2+469  {l}^4\right)+...
\label{Birk}
\end{align}
\end{lemm}
\begin{proof}
We want to obtain the energy $\mfh$ as a function of the value of the imaginary action $\mfj$ and the value of the angular momentum $\mfl$. We will express this result as a Taylor series in $\mfj$ and $\mfl$. This is a modification of the classical Birkhoff normal form adapted to focus-focus points. We start by calculating the expansion of $\mfJ(\mfh,\mfl)$ around the origin and afterwards we invert it, forgetting the dependence on $\mfl$.  Following Dullin \cite{Du}, we can obtain the series expansion of $\mfJ(\mfh,\mfl)$ either from \eqref{expJ2}, the coefficient  of the logarithmic term in Lemma \ref{act3}, or  directly from the definition \eqref{eqJ}. It is obviously faster to just use the first option but we will briefly illustrate the second option in order to get a better insight into the problem.

The imaginary $\alpha$ cycle vanishes when we approach the singular fibre. By the residue theorem of complex analysis, this means that only the values of the integrand close to $(\mfl,\mfh)=(0,0)$ matter. So we can expand the integrand of \eqref{eqJ} by writing it in scaled coordinates and making again the substitutions $\mfh \mapsto \mfh \varepsilon$, $\mfl \mapsto  \mfl \varepsilon$, which leaves
\begin{align}
 &-\dfrac{ \sqrt{2} \mfh(3\mfp_2-4)}{{\mfp_2}(2-\mfp_2)^{3/2}}\varepsilon 
+\dfrac{\sqrt{2} \mfh \mfl (5{\mfp_2}^2-11\mfp_2+8)}{{\mfp_2}^2(2-\mfp_2)^{5/2}} \varepsilon^2 \nonumber \\[0.2cm]
&+ \dfrac{2\mfh^3(3{\mfp_2}^2-10\mfp_2+8)+\mfh \mfl^2(-14{\mfp_2}^3+44{\mfp_2}^2-65\mfp_2+36)}{\sqrt{2}{\mfp_2}^3(2-\mfp_2)^{7/2}} \varepsilon^3 + ... 
\label{expans}
\end{align}
$\mfJ$ is the residue of the integrand at the pole $p_2=0$. So by calculating the residue of \eqref{expans} at the origin, we recover \eqref{expJ}. Once we have the series expansion of $\mfJ(\mfl,\mfh)$ we only need to invert it regarding $\mfl$ as a constant. In other words, we fix the values of $\mfj=\mfJ(\mfl,\mfh)$ and $\mfl$, and then solve term by term to get $\mfh$ as a function of $\mfj$ and $\mfl$. This inversion gives us 
\begin{align}
\mfB(\mfj,\mfl)&=\dfrac{1 }{2 }\mfj +\dfrac{ 1}{16}{\mfj}{\mfl} -\dfrac{5}{512}\mfj\left({\mfj}^2+{\mfl}^2\right) +\dfrac{ 1}{4096 } \mfj\left(15 {\mfj}^2{\mfl}+11  {\mfl}^3\right)\nonumber \\
& -\dfrac{ \mfj\left(393 {\mfj}^4+990 {\mfj}^2 {\mfl}^2+469  {\mfl}^4\right)}{524288}+...
\label{Birks}
\end{align} and reverting the scaling of variables we obtain \eqref{Birk}.
\end{proof}

\begin{re} The dependence of the imaginary action $J(l,h)$ on $h$ comes from a global linear factor $h$ and through nonlinear dependence on $h^2$ as we see in \eqref{eqJ}, so $J$ is an odd function of $h$: $J(l,-h) = -J(l,h)$. This is also reflected in the normal form $B$. This property is kept when we do the inversion, so we have $B(-j,l)=-B(j,l)$.
\end{re}

We define $z:=j+il$, where $i$ is the imaginary unit. In section \ref{sec:deftaylor} we saw that the area integral
\begin{equation*}
\mcA (z) = 2\pi I(l,B(j,l))
\end{equation*} satisfies 
\begin{equation*}
 \mathcal{A}(z) = \mathcal{A}_0 - \Re(z \log z-z) + S(l,j).
\end{equation*}  

\begin{theo}
\label{inv1}
The series expansion of the area integral $\mcA(z)$ of the spin-oscillator as a function of the value $j$ of the imaginary action and $l$ of the angular momentum, where $z=j+il$, is
\begin{equation}
 \mathcal{A} (z) = 2  \pi\lambda  -j \log |z| +j+l \arg (z)+    S(j,l)
 \label{act7}
\end{equation} where the Taylor series invariant $S(j,l)$ is
\begin{align}
S(j,l) &=  (5 \log 2+ \log \lambda) j +\dfrac{\pi}{2} l+ \dfrac{1}{4\lambda} jl - \dfrac{1}{768\lambda^2}j (34 {j}^2+39 {l}^2 ) + \dfrac{1}{1536 \lambda^3}j(23 {l}^3 + 34l {j}^2 ) \nonumber \\
&- \dfrac{1}{2621440 \lambda^4}j (10727 {j}^4+30620 {j}^2{l}^2+13505 {l}^4) + ...
\label{invar}
\end{align}
\end{theo}

\begin{proof}
We only need to substitute the normal form from Lemma \ref{lact} into the action expansion that we have found in Lemma \ref{act3}.
\end{proof}

It is somehow  amazing that even though in the action expansion in Lemma \ref{act3} rational functions and square roots appear, they all suitably combine into a polynomial in $j$ and $l$ as predicted by the theory. A few remarks are in order:

\begin{re}
For the particular case $\lambda=\mu=1$, the first-order terms in this expansion were already calculated in Pelayo \&  V\~u Ng\d{o}c \cite{PV3} and Babelon \& Douçon \cite{BD}. They coincide with the ones in \eqref{invar}.
\end{re}

\begin{re}
We see that the constant $\mu$ does not appear in the Taylor series invariant, because  the action $l$ corresponding to the coordinates $(q_1,p_1)$  is already scaled. In other words, since the Euclidean plane is non-compact and both the plane and the rotation action are isotropic, the coordinates scale in a natural way so that the constant $\mu$ plays no role. The sphere is compact, so $\lambda$ must play a role, but it is nothing else than a simple rescaling. The only special term is $\log \lambda$ in the first-order part, but it comes from scaling the term $\log|z|$ in $\mcA(z)$. 
\end{re}

\begin{re}
The value of the constant $\mathcal{A}_0$ corresponds, up to a factor $2\pi$, to the height invariant defined in \S \ref{sec:defheight}. 

In the case of the coupled spin-oscillator, the focus-focus singularity has the critical value $c=(0,0)$. If $H=h=0$, it means that $q_2$ is independent of $p_2$, taking the value $\pi/2$ when $p_2$ goes from $\zeta_2=0$ to $\zeta_3=2\lambda$ and $-\pi/2$ in the opposite direction. Therefore  $\mathcal{A}_0 = \mathcal{A}(0,0)$ is given by
$$\mathcal{A}_0  =  \oint q_2\, dp_2 =  \dfrac{\pi}{2} \int_{0}^{2\lambda} dp_2 - \dfrac{\pi}{2} \int_{2\lambda}^{0} dp_2 =  \pi (2\lambda-0) = 2\pi \lambda,$$
which essentially is the area of a rectangle. It coincides with 
$$2\pi h_1 = 2\pi \text{Vol}(Y^-) = 2\pi \text{Vol}(Y^+) =  \left( \dfrac{\pi}{2}+ \dfrac{\pi}{2} \right) (2\lambda-0)=2\pi \lam,$$ where $Vol$ denotes the symplectic volume, since $H$ will be positive for half of the possible values of $q_2$ and negative for the rest. 
\end{re}

The symmetry of the foliation with respect to the transformations $T_1$, $T_2$ has the following effect on the symplectic invariant $S(j,l)$:

\begin{theo}
The Taylor series invariant $S(j,l)$ of the spin-oscillator has only odd powers of $j$.
\label{noeven}
\end{theo}
\begin{proof}
The invariant $S(j,l)$ is a property of the foliation $\mcF$ induced by $(L,H)$. The foliation $\mcF$ is invariant under the symplectic transformations $T_1$, $T_2$ of $M$ defined in \eqref{transf}. These transformations act on the functions as
$$ L \mapsto L \quad H \mapsto -H \quad J \mapsto -J \quad B \mapsto -B$$
and on the variables as
$$ l \mapsto l \quad h \mapsto -h \quad j \mapsto -j.$$
From \eqref{act7} we see that the symplectic invariant is given by 
\begin{equation}
S(j,l) = \mcA(z) - 2\pi \lam + j \log |z|-j - l \arg(z).
\label{SS}
\end{equation} Let us assume $l>0$ for simplicity. Now apply $T_1$ or $T_2$ to \eqref{SS}.  We know that $j$ will change sign, $\arg(z)$ will become $-\pi- \arg(z)$ and from \eqref{I} we see that  $\mcA(z)$ will transform into $4\pi \lam - \mcA(z)$. So we obtain
\begin{align*}
S(-j,l)&= 4 \pi \lam - \mcA(z) - 2\pi \lam - j \log |z| + j + l\pi +l \arg(z) \\
&= -S(j,l) + l\pi.
\end{align*} Thus we obtain $S(j,l)-\frac{\pi}{2}l = -S(-j,l) + \frac{\pi}{2}l$, i.e.\ $S(j,l) -\frac{\pi}{2}l $ is an odd function of $j$, which implies that the Taylor series of $S(j,l)$ cannot have terms with even powers of $j$. 
\end{proof}


\section{Period and rotation number}

Using the same techniques we can calculate the expansions of other dynamical quantities of interest, such as the period $T$ and the rotation number $W$. Consider a regular fibre of the semitoric system $(M,\omega,(L,H))$ defined by \eqref{syst}, which is diffeomorphic to $\T^2=\mbS^1 \times \mbS^1$. Consider also the reduction of the system by the $\mbS^1$-action induced by $L$, with fibres diffeomorphic to $\mbS^1$. Then the reduced period $T$ essentially refers to the period of the reduced system and the rotation number $W$ is the quotient of the two rotation frequencies on the 2-torus fibre.

Let $I(l,h)$ be the action integral of the reduced system. The period is defined as
\begin{equation}
T(l,h) := 2\pi \dfrac{\partial I}{\partial h}(l,h)
\label{per1}
\end{equation}
and we denote its scaled counterpart by $\mfT :=\frac{1}{\sqrt{\lambda \mu}}T$. It is also useful to know the expression of the period as a function of the action values $j$ and $l$, which we will denote by $\That(j,l) := T(l,B(j,l))$.

\begin{lemm}
The period $T(l,h)$ has the expansion
\begin{align}
\dfrac{T(l,h)}{\sqrt{\lam \mu}} &=  \dfrac{1}{4\lam} \dfrac{ l^3+6\lam \mu lh^2}{l^2+4\lam \mu h^2} - \dfrac{1}{256\lam^2}\dfrac{21l^6+364l^4h^2+1680l^2h^4+2496h^6}{(l^2+4 \lam \mu h^2)^2} \nonumber\\
&+ \dfrac{\pi}{6144\lam^3} \dfrac{185l^9+ 6084\lam \mu l^7h^2+ 54288\lam^2 \mu^2 l^5h^4+193728\lam^3 \mu^3 l^3 h^6 +244224\lam^4 \mu^4 l h^8}{(l^2+4 \lam \mu h^2)^3}+...\nonumber\\
&+\dfrac{T^\alpha(l,h)}{2\pi\sqrt{\lam \mu}} \log \dfrac{32\lambda}{\sqrt{l^2+4\lam \mu h^2}}, 
\label{per2}
\end{align}
around the focus-focus point, where 
\begin{align*}
\dfrac{T^\alpha(l,h)}{2\pi\sqrt{\lam \mu}} &= 1-\dfrac{l}{8\lambda} + \dfrac{3}{256\lambda^2} (3l^2+20\lam \mu h^2) - \dfrac{5}{2048\lambda^3}( 5l^3+84 \lam \mu l h^2) \\
&+ \dfrac{35}{262144\lambda^4} (35l^4+1080\lam \mu l^2h^2+1584\lam^2 \mu^2 h^4) +...
\end{align*}
\label{per3}
\end{lemm}
\begin{proof}
There are two ways to calculate the period. The first and simplest is to apply the definition \eqref{per1} to the series expansion \eqref{act2} of $I(l,h)$, which leads to the desired result \eqref{per2}.

The second way, however, can give us a better insight of the meaning of the period expansion. Let us switch to scaled coordinates. Using  expression \eqref{act1} we have
\begin{equation*}
\mfT(\mfl,\mfh)= 2 \pi \dfrac{\partial \mfI}{\partial \mfh} (\mfl,\mfh)= \oint_{\beta_{\mfl,\mfh}} \dfrac{\partial}{\partial \mfh}  \arccos \left( \dfrac{\sqrt{2}\, \mfh }{\sqrt{-\mfp_2 (\mfp_2-\mfl) (\mfp_2-\mfl-2)}} \right)d\mfp_2 = -2\oint_{\beta_{\mfl,\mfh}} \dfrac{d\mfp_2}{\mfw},
\end{equation*} where $w$ is as in \eqref{curv}. The period is therefore an elliptic integral over the same elliptic curve $\Ga_{\mfl,\mfh}$ as $\mfI(\mfl,\mfh)$. We finally have
\begin{equation*}
\mfT(\mfl,\mfh) = 4 \int_{\zeti_2}^{\zeti_3} \dfrac{d\mfp_2}{\mfw} = 4\,\mcN_A = \dfrac{4\sqrt{2}  }{\sqrt{\zeti_3-\zeti_1}} K(\mfk),
\end{equation*} where in the last step we used \eqref{Ia}. Using the expansions of the proof of Lemma \ref{act3} we obtain
\begin{align}
\mfT(\mfl,\mfh) &= \dfrac{1}{4} \dfrac{ \mfl^3+6\mfl\mfh^2}{\mfl^2+4\mfh^2} - \dfrac{1}{256}\dfrac{21\mfl^6+364\mfl^4\mfh^2+1680\mfl^2\mfh^4+2496\mfh^6}{(\mfl^2+4\mfh^2)^2} \nonumber\\
&+ \dfrac{1}{6144} \dfrac{185\mfl^9+ 6084\mfl^7\mfh^2+ 54288 \mfl^5\mfh^4+193728 \mfl^3 \mfh^6 +244224\mfl\mfh^8}{(\mfl^2+4\mfh^2)^3}+...\nonumber\\
&+\dfrac{\mfT^\alpha(\mfl,\mfh)}{2\pi} \log \dfrac{32}{\sqrt{\mfl^2+4\mfh^2}}, 
\end{align}
where 
\begin{align*}
\mfT^\alpha(\mfl,\mfh) &= 1-\dfrac{ \mfl}{8} + \dfrac{3}{256} (3\mfl^2+20\mfh^2) - \dfrac{5}{2048}( 5\mfl^3+84\mfl \mfh^2) \\
&+ \dfrac{35}{262144} (35\mfl^4+1080\mfl^2\mfh^2+1584\mfh^4) +...
\end{align*}
Reverting the scaling of variables we obtain \eqref{per2}.
\end{proof}

By substituting $h$ by the modified Birkhoff normal form of Lemma \ref{lact} we obtain the series expansion of $\That(j,l)$.

\begin{co}
The period $\That(j,l)$ as a function of the values $j$ of the imaginary action and $l$ of the angular momentum has for $z = j + i l$ the expansion
\begin{align*}
\That&(j,l)=\dfrac{\sqrt{\lambda \mu}}{2\lambda} l - \dfrac{\sqrt{\lambda \mu}}{128\lambda^2} (34j^2+21l^2) +\dfrac{5\sqrt{\lambda \mu}}{3072\lambda^3}l (120j^2+37l^2)+... \\
&+ (\log|z|-\log (32\lambda)) \left( -2\sqrt{\lambda \mu} + \dfrac{\sqrt{\lambda \mu}}{4\lambda}l - \dfrac{\sqrt{\lambda \mu}}{128\lambda^2} (15j^2+9l^2)+\dfrac{\sqrt{\lambda \mu}}{1024\lambda^3} l(75j^2+25l^2)+... \right)
\end{align*}
around the focus-focus point.
\end{co}

Similarly, we define the rotation number as
\begin{equation}
W(l,h) := - \dfrac{\partial I}{\partial l}(l,h)
\label{w1}
\end{equation} 
and its scaled counterpart as $\mfW := W$. 

We also denote the rotation number as a function of the action values $j$ and $l$ by $\What(j,l) := W(l,B(j,l))$.

\begin{lemm}
The rotation number $W(l,h)$ has the expansion
\begin{align}
2\pi W(l,h) &= -\dfrac{\pi}{2} + \arctan \left(\dfrac{l}{2\sqrt{\lam \mu}\,h} \right) -\dfrac{\sqrt{\lam \mu}}{4\lambda}\dfrac{3l^2+8\lam \mu h^2}{l^2+4\lam \mu h^2}\nonumber\\&+\dfrac{\sqrt{\lam \mu}}{128\lambda^2} \dfrac{lh(51l^4 + 392\lam \mu l^2h^2 +816\lam^2\mu^2 h^4)}{(l^2+4\lam \mu h^2)^2}\nonumber\\
&- \dfrac{\sqrt{\lam \mu}}{1536\lam^3} \dfrac{h(315 l^8 + 4434\lam \mu l^6h^2+ 22872 \lam^2 \mu^2 l^4h^4 + 49248 \lam^3 \mu^3 l^2 h^6 +36224 \lam^4 \mu^4 h^8) }{(l^2+4\lam \mu h^2)^3}+...\nonumber\\
&+W^\alpha(l,h) \log \dfrac{32\lam}{\sqrt{l^2+4 \lam \mu h^2}}, 
\label{w2}
\end{align}
around the focus-focus point, where 
\begin{align*}
\dfrac{W^\alpha(l,h)}{\sqrt{\lam \mu}} &=\dfrac{h}{4\lambda} - \dfrac{9lh }{64\lambda^2} + \dfrac{5h}{1024\lambda^3}(15l^2+28\lam \mu h^2) - \dfrac{175lh}{32768\lambda^4}( 7l^2+36\lam \mu h^2 )  \\
&+ \dfrac{63h}{1048576\lambda^5} (315l^4+3080\lam \mu l^2 h^2 +2288\lam^2 \mu^2 h^4) +...
\end{align*}
\end{lemm}
\begin{proof}
The proof of this Lemma is completely analogous to the one of Lemma \ref{per3}. We can  either apply the definition \eqref{w1} directly to the series expansion \eqref{act2} of $I(l,h)$, which leads to the result \eqref{w2}, or work out the explicit integrals in order to get an idea of the structure  of the rotation number. Switching to scaled variables and using expression \eqref{act1} we have
\begin{align*}
2\pi\mfW (\mfl,\mfh) &= - \dfrac{\partial \mfI}{\partial \mfl} (\mfl,\mfh)= - \oint_{\beta_{\mfl,\mfh}} \dfrac{\partial}{\partial \mfl} \arccos \left( \dfrac{\sqrt{2}\, \mfh }{\sqrt{-\mfp_2 (\mfp_2-\mfl) (\mfp_2-\mfl-2)}} \right)d\mfp_2 \\
&= -\mfh \int_{\zeti_2}^{\zeti_3} \left( \dfrac{1}{\mfp_2-\mfl} + \dfrac{1}{\mfp_2-\mfl-2} \right) \dfrac{d\mfp_2}{\mfw} =-2 \left( \mcN_{B,\mfl} + \mcN_{B,\mfl+2}\right).
\end{align*} 
By substituting now \eqref{Ib} we obtain
\begin{equation*}
2\pi\mfW (\mfl,\mfh) = \frac{-2\sqrt{2} \mfh }{\sqrt{\zeti_3-\zeti_1}}  \left( \dfrac{1}{(\zeti_3-\mfl)} \Pi(\mfn_\mfl,\mfk) + \dfrac{1}{(\zeti_3-\mfl-2)} \Pi(\mfn_{\mfl+2},\mfk)\right). 
\end{equation*} 
Using the expansions of the proof of Lemma \ref{act3} we get
\begin{align}
2\pi\mfW(\mfl,\mfh) &= -\dfrac{\pi}{2} + \arctan \left(\dfrac{\mfl}{2\mfh} \right) - \dfrac{3\mfh}{4}+\dfrac{\mfh^3}{\mfl^2+4\mfh^2}+\dfrac{\mfh}{128} \dfrac{51\mfl^5 + 392 \mfl^3\mfh^2 +816\mfl \mfh^4}{(\mfl^2+4\mfh^2)^2}\nonumber\\
&- \dfrac{1}{1536} \dfrac{315 \mfl^8\mfh + 4434\mfl^6\mfh^3+ 22872 \mfl^4\mfh^5 + 49248 \mfl^2 \mfh^7 +36224\mfh^9 }{(\mfl^2+4\mfh^2)^3}+...\nonumber\\
&+\mfW^\alpha(\mfl,\mfh) \log \dfrac{32}{\sqrt{\mfl^2+4\mfh^2}}, 
\end{align}
where 
\begin{align*}
\mfW^\alpha(\mfl,\mfh) &=\dfrac{\mfh}{4} - \dfrac{9\mfl\mfh }{64} + \dfrac{5\mfh}{1024}(15\mfl^2+28\mfh^2) - \dfrac{175\mfl\mfh}{32768}( 7\mfl^2+36\mfh^2 )  \\
&+ \dfrac{63\mfh}{1048576} (315\mfl^4+3080\mfl^2 \mfh^2 +2288\mfh^4) +...
\end{align*}
Reverting the scaling of variables we obtain \eqref{w2}. 
\end{proof}

By substituting $h$ by the modified Birkhoff normal form of Lemma \ref{lact} we obtain the series expansion of $\What(j,l)$.
\begin{co}
The rotation number $\What(j,l)$ as a function of the values $j$ of the imaginary action and $l$ of the angular momentum has, for $z = j + i l$, the expansion
\begin{align}
2\pi\What(j,l)&=-\dfrac{\pi}{2} - \arg(z) -\dfrac{j}{4\lambda} + \dfrac{17jl}{128\lambda^2}- \dfrac{7}{6144\lambda^3}j(34j^2+57 l^2)+...  \nonumber \\
&- (\log|z|-\log (32\lambda)) \left( \dfrac{j}{8\lambda} - \dfrac{7jl}{128\lambda^2} + \dfrac{1}{1024\lambda^3}j(15j^2+26l^2)+... \right)
\label{rot}
\end{align}
around the focus-focus point. 
\end{co}

Dullin \& \vungoc  \cite{DV} showed that the rotation number must have the form 
$$ 2 \pi \What(j,l) = -A(j,l) \Re (\log z) - \Im (\log z) + \varsigma (j,l)$$ where $A(j,l)$ is the smooth function 
$$A(j,l) = \frac{ \frac{\partial B}{\partial l}  (j,l)}{ \frac{\partial B}{\partial j}(j,l)}$$ 
and $\varsigma(j,l)$ is a smooth and single-valued function near the origin. Comparing with \eqref{rot}, we see that it has the expected form, where
$$ A(j,l) =  \dfrac{\frac{\partial B}{\partial l}(j,l)}{ \frac{\partial B}{\partial j}(j,l)}= \dfrac{j}{8\lam} - \dfrac{7jl}{128\lam^2} + \dfrac{1}{1024\lam^3}(15j^3+26jl^2)-\dfrac{1}{65536\lam^4}(1005j^3l + 791jl^3)+... $$ and
$$
\varsigma (j,l) = (\log 32\lam) A(j,l) -\dfrac{\pi}{2} -\dfrac{j}{4\lam} + \dfrac{17jl}{128\lam^2}- \dfrac{7}{6144\lam^3}(34j^3+57j l^2)+...$$


\section{Vanishing twist and superintegrability}

The variation of the rotation number with respect to the angular momentum is the \emph{twist}
\begin{equation}
\mathcal{T}(l,h) := \dfrac{\partial W}{\partial l}(l,h)
\label{T1}
\end{equation} 
and we set $\hat{\mathcal{T}}(j,l) := \mcT(l,B(j,l))$. The twist around focus-focus points has been studied for example in Dullin \& Ivanov \cite{DI} and Dullin \& V\~u Ng\d{o}c \cite{DV}. In the second paper it is shown that, when the focus-focus point is loxodromic, i.e., the Hessian of the function $H$ has four complex eigenvalues with non-zero real and imaginary parts, then there exist a regular torus with vanishing twist for each value of $h$ close to the critical one and all other tori with the same $h$ have non-vanishing twist. However, it was not clear what the behaviour would be in other situations.

In our case though, the eigenvalues of the Hessian of the Hamiltonian function $H$ at the fixed point are $\pm 1$, thus real. We have the following result:

\begin{lemm}
The twist $\mcT(l,h)$ has the expansion
\begin{align*}
2\pi\dfrac{\mcT(l,h)}{\sqrt{\lam \mu}} &= -\dfrac{2h}{l^2+4\lam \mu h^2} - \dfrac{1}{4\lam}\dfrac{lh(l^2+12\lam \mu h^2)}{(l^2+4 \lam \mu h^2)^2} \\
&+\dfrac{h}{128\lam^2} \dfrac{69l^6+772\lam \mu l^4h^2+2544\lam^2 \mu^2 l^2h^4+3264\lam^3 \mu^3 h^6}{(l^2+4\lam \mu h^2)^3}+... \\
&+ \left( -\dfrac{9h}{64\lam^2}+\dfrac{75lh}{512\lam^3}-\dfrac{525h}{32768\lam^4}(7l^2+12\lam \mu h^2)+...\right) \log \dfrac{32\lam}{\sqrt{l^2+4\lam \mu h^2}}
\end{align*}
around the focus-focus point. 
\end{lemm}
\begin{proof}
Directly from \eqref{w2} and \eqref{T1}.
\end{proof}

\begin{co} The twist $\hat{\mcT}(j,l)$ as a function of the values $j$ of the imaginary action and $l$ of the angular momentum has the expansion
\begin{align*}
2\pi\hat{\mcT}(j,l) = \dfrac{1}{|z|^2}&\left(-j -\dfrac{jl}{4\lambda} + \dfrac{1}{128\lambda^2}j(23j^2+35l^2) - \dfrac{1}{128\lambda^3}jl(26j^2 + 27l^2)+...\right. \\
 &+ \left.(\log |z|-\log (32\lambda)) \left( \dfrac{1}{128\lambda^2}j(9l^2+9j^2) - \dfrac{1}{512\lambda^3}jl(33j^2 + 33l^2) + ... \right) \right),
\end{align*}  where $z=j+il$.
\end{co}

We see that the expansion of the twist is an odd function of $j$. More precisely,

\begin{theo}
The twist of the spin-oscillator vanishes when the Hamiltonian vanishes, i.e.\ $\mcT=0$ on $H^{-1}(0)$.
\label{theotwis}
\end{theo}

\begin{proof}
It is more convenient to work with the twist $\hat{\mcT}(j,l)$ as a function of the values $j$ of the imaginary action and $l$ of the angular momentum. It can be obtained directly from the action integral $\mcA(z)$ via
$$ 2\pi\hat{\mcT}(j,l) =2\pi \dfrac{\partial \What}{\partial l} (j,l) = -\dfrac{\partial^2 \mcA}{\partial l^2} (j,l).$$

Let us now look at the structure of $\mcA(z)$ in \eqref{act7}. The first term, $2\pi \lambda$, is a constant and therefore vanishes after the derivative. The second term, $- j \log |z|$, is odd in $j$ and therefore, after taking two derivatives with respect to $l$, it will still be odd in $j$. More precisely,
$$ -\dfrac{\partial^2}{\partial l^2} \left( - j \log |z| \right) = 	j\dfrac{j^2-l^2}{|z|^4}.$$

The third term, $j$, vanishes after taking the derivative. The fourth term, $l \arg(z)$, is odd in $j$ and it remains so after two derivatives:
$$ -\dfrac{\partial^2}{\partial l^2} \left( l \arg(z) \right) = -\dfrac{2j^3}{|z|^4}.$$ 

From the proof of Theorem \ref{noeven} we see that the last term, $S(j,l)$, is odd in $j$ up to a linear factor in $l$ that vanishes after differentiating twice. This means that $\hat{\mcT}(j,l)$ is an odd function of $j$ and in particular, it vanishes if $j=0$. Since $J$ vanishes when $h=0$ we conclude that $\mcT(l,0)=0$, which is what we wanted to see.
\end{proof}

The fact that the twist vanishes at $h=0$ for all values of $l$ is thus a consequence of the discrete symmetry of the foliation. A possible way to interpret this property is in terms of {\em superintegrability} (see Fassò \cite{Fa} for an overview). When $h=0$, a new conserved quantity appears. This additional integral has two expressions given by 
$$
K_{uv}(x,y,z,u,v):=\arg \left(u+iv\right), \quad
K_{xy}(x,y,z,u,v):=\arg (x+iy) \,.
$$ 
On the level set $H=0$ these two functions are dependent, so they really only define one additional integral. Applying Mischenko-Fomenko's theorem of `noncommutative integrability' (Nekhoroshev \cite{Ne}, Mischenko \& Fomenko \cite{MF}) --- the analog of Liouville-Arnold's theorem for superintegrable systems --- we obtain a fibration in terms of 1-dimensional tori. 
\begin{lemm}
If the Hamiltonian vanishes, $H=0$, then the spin-oscillator has a third constant of motion, $K_{uv}(x,y,z,u,v)=\arg(u+iv)$
and $K_{xy}(x,y,u,v) = \arg(x+iy)$. Together they define an integral that is smooth almost everywhere, 
and all orbits of the system are closed.
\end{lemm}
\begin{proof}
The symplectic form of the system \eqref{syst} is $\omega = \lambda\, \omega_{{\mathbb S}^2} \oplus \mu\, \omega_{\mathbb{R}^2}$. Therefore,
$$ \{H,K_{uv}\} =\dfrac{\partial H}{\partial u}\dfrac{\partial K_{uv}}{\partial v}\{u,v\}+\dfrac{\partial H}{\partial v}\dfrac{\partial K_{uv}}{\partial u}\{v,u\} = \dfrac{x}{2}\dfrac{u}{u^2+v^2}\dfrac{1}{\mu} + \dfrac{y}{2}\dfrac{(-v)}{u^2+v^2}\dfrac{(-1)}{\mu} = \dfrac{1}{\mu} \dfrac{H}{u^2+v^2}$$ and consequently $\{H,K_{uv}\}=0$ if $H=0$. By a similar calculation $\{H, K_{xy} \} = \frac{\lambda H z}{ x^2 + y^2}$ which again vanishes when $H=0$. 
Each expression for the additional integral has a singularity at the origin, but they are simultaneously singular only when $x=y=u=v=0$, i.e.\ at the equilibrium points $z=\pm1$ of the system. For all other orbits we thus have an additional smooth third integral, and hence all orbits of the system for $H=0$ are closed.
\end{proof}

Note that $K_{uv}$ and $L$ are not in involution. More precisely,
$$\{L,K_{uv}\}=\dfrac{\partial L}{\partial u}\dfrac{\partial K_{uv}}{\partial v}\{u,v\}+\dfrac{\partial L}{\partial v}\dfrac{\partial K_{uv}}{\partial u}\{v,u\} = \mu u \dfrac{u}{u^2+v^2} \dfrac{1}{\mu} + \mu v \dfrac{(-v)}{u^2+v^2} \dfrac{(-1)}{v} =1,
$$
and similarly $\{ L, K_{xy} \} = -1$.

The geometric meaning of the additional integral is that the orbits are restricted to lines through the origin in the 
$(x,y)$-plane and the $(u,v)$-plane. More precisely, introduce a single new variable $a$ instead of $x$ and $y$,
and a single new variable $b$ instead of $u$ and $v$, with the relations 
$x = a \cos\theta$, $y=a \sin\theta$, 
$u = -b \sin\theta$, $v= b \cos\theta$ for some angle $\theta$. 
Then $H=0$ and $J = \frac{1}{2}\mu b^2 + \lambda( z - 1)$, and it implies that $\dot b = -\frac{1}{2}a$ 
and $-2 \ddot b = \dot a = -\frac{1}{2}b z$ where $z$ can be expressed in terms of $b$ using the integral $J$.
This second-order differential equation for $b$ has a quartic integral, and once again, all orbits are closed when $H=0$.


\section{Calculation of the twisting-index invariant}
The coupled spin-oscillator \eqref{syst} is a semitoric system with one focus-focus singularity located at $m=(0,0,1,0,0)$. As mentioned in Remark  \ref{twi1}, the twisting-index invariant is completely specified by finding a representative of the polygon invariant with index $k=0$. From there, the indices for the rest of the polygons in the equivalence class can be reconstructed by knowing that $\Z_2$ does not act on the index and that $\mcG \simeq \Z$ acts by addition.

More precisely, the action  of $ \Z_2 \times \mcG$ on $W\text{Polyg}(\R^2) \times \Z$ is 
\begin{equation}
(\epsilon', T^{k'}) \star (\De, b_{\ka}, \epsilon, k) = (t_u(\De), b_{\ka}, \epsilon'\epsilon, k+k'),
\label{act11}
\end{equation} where $u = (\epsilon-\epsilon')/2$.  As a consequence, an orbit of this action is completely specified by finding a weighted polygon $(\De,b_{\ka},\epsilon)$ with index $k=0$. The rest of the elements of the orbit can be calculated from \eqref{act11}.

\begin{theo}
The twisting-index invariant of the coupled spin-oscillator system is the orbit generated by the action of $\Z_2 \times \mcG$ on the element $(\De,b_{\ka},\epsilon,k)$, where $\De$ is the polygon depicted in Figure \ref{twistfig}, $b_{\ka}=\{(0,y)\;|\;y \in \R\}$, $\epsilon=+1$ and $k=0$. 
\label{inv2}
\end{theo}

\begin{proof}

Let $W \subseteq M$ be a neighbourhood of the focus-focus singularity $m=(0,0,1,0,0)$ and $V:=F(W)$. Let us also make the choice $\epsilon=+1$. We also know $F(m)=(L,H)(m)=(0,0)$, so in particular $\ka=0$ and $b^{\epsilon}_{\ka}=b_0^{+1}$. In section \S \ref{sec:deftwistingIndex} we have seen that there is a unique smooth function $H_p:F^{-1}(V \backslash b_0^{+1}) \to \R$ that extends to a continuous function on $F(V)$, satisfies $\lim_{q \to m} H_p (q) =0$ and whose Hamiltonian vector field coincides with the vector field $\mcX_p$ from \eqref{Xp}. In other words, $H_p$ must be of the form $f \circ \Phi$, where  $f=f(z)$ is a function whose derivatives  are $\partial_i f = \tau_i/2\pi$, $i=1,2$, the functions $\tau_i$ are as in \eqref{taus} and $\Phi= \phi \circ F$ as in \S \ref{sec:deftwistingIndex}. This implies that $f$ must be of the form 
$$2\pi f(z) = S(z)-\Re(z \log z -z) + \mbox{ const.},$$
where $\log$ has the cut along the positive real axis.

Comparing with \eqref{inv5} and \eqref{act7} we see that $f(z) := \frac{1}{2\pi} (\mcA(z)-\mcA_0)$ and $\phi(l,h) := \phi(J(l,h),l)$ have the desired form, thus by uniqueness, we must have $H_p = f \circ \phi \circ F$, which can be defined continuously in all $M$. The privileged momentum map will then be $\nu = (L,H_p)$ and it will coincide with one of the polygons of the polygon invariant. Even though we only know $S$ up to finite order, it is enough to identify which is the polygon with index $k=0$, as shown in Figure \ref{twistfig}.

It is a polygon with vertices at $(-2\lam,-\lam)$ and $(0,\lam)$. The first vertex corresponds to the elliptic-elliptic singularity. The inner point $(0,0)$ is the image of the focus-focus singularity. The polygon $\De$, together with $b_0^{+1}$ and $\epsilon=+1$, is a weighted polygon of the polygon invariant, so it is also the image of a momentum map $\mu: M \to \De$. For this momentum map we will have $\mu = \nu$ and using \eqref{twistwis} we conclude that $k=0$. The association of $k=0$ to the polygon in Figure \ref{twistfig} completely determines the twisting-index invariant.
\end{proof}

 \begin{figure}[ht!]
\centering
\subfloat[]{
  \includegraphics[width=0.46\textwidth]{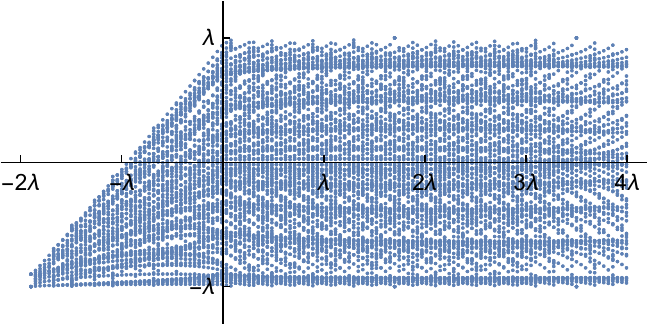}
}
\subfloat[]{
  \includegraphics[width=0.45\textwidth]{Polyg.pdf}
}
 \caption{\small Numerical plot of the image of the privileged momentum map of the coupled spin oscillator (a) and the corresponding weighted polygon $\De$ of the polygon invariant (b). This polygon has thus index $k=0$. The numerical plot is made by sampling the phase space with 73700 points and using $S$ up to second order. The image of the elliptic-elliptic singularity is $(-2\lam,-\lam)$ and the image of the focus-focus singularity is at $(0,0)$.}
 \label{twistfig}
\end{figure}

The polygon symplectic invariant of the coupled spin-oscillator has been calculated by Pelayo \& \vungoc in Figure 5 of \cite{PV3}. We see that it coincides with the one in Figure \ref{twistfig} except for a horizontal translation by $\lambda$ and a vertical translation by $\lambda$. The horizontal translation is caused by our definition of the function $L$ in \eqref{syst} which differs from theirs by $\lambda$ since we work with $(L,H)(m)=(0,0)$. The vertical translation is due to the requirement that $H_p$ tends to $0$ as we approach the focus-focus singularity $m$, which amounts to substracting $\mcA_0$ in the definition of the function $f$ of the proof of Theorem \ref{inv2}. 

Once we have a weighted polygon of the polygon invariant with the associated twisting index $k=0$, the twisting-index invariant is completely determined, i.e., we can calculate the index associated to all other weighted polygons of the polygon invariant. In Figure \ref{twistbig} some of them are displayed.



\bibliographystyle{alpha}
\bibliography{PhD}
\vspace{0.5cm}
\textbf{Jaume Alonso}\\
Department of Mathematics and Computer Science\\
University of Antwerp\\
Middelheimlaan 1\\
2020 Antwerp, Belgium\\
\textit{E-mail:} \texttt{jaume.alonsofernandez@uantwerpen.be}

\vspace{0.5cm}
\textbf{Holger R. Dullin}\\
School of Mathematics and Statistics\\
University of Sydney\\
Camperdown Campus\\
Sydney, NSW 2006, Australia \\
\textit{E-mail:} \texttt{holger.dullin@sydney.edu.au}

\vspace{0.5cm}
\textbf{Sonja Hohloch}\\
Department of Mathematics and Computer Science\\
University of Antwerp\\
Middelheimlaan 1\\
2020 Antwerp, Belgium\\
\textit{E-mail:} \texttt{sonja.hohloch@uantwerpen.be}

\end{document}